\documentclass[EJP]{ejpecp} 

\usepackage[utf8]{inputenc}

\usepackage[english,notheorems]{package_RFb}

\usepackage[numbers]{natbib}
\usepackage{hyperref}

\SHORTTITLE{The stepping stone model in a random environment} 

\TITLE{The stepping stone model in a random environment and the effect of local heterogneities on isolation by distance patterns} 

\AUTHORS{%
  Rapha\"el~Forien\footnote{CMAP, \'{E}cole polytechnique, 91128, Palaiseau Cedex, France; INRA BioSP, 84000 Avignon, France; \newline R.F. was supported in part by the chaire Modélisation Mathématique et Biodiversité of Veolia Environnement-École Polytechnique-Museum National d’Histoire Naturelle-Fondation X.
    \BEMAIL{raphael.forien@normalesup.org}}}


\AMSSUBJ{60F17; 60G99; 60H15; 60J25} 

\SUBMITTED{July 5, 2018} 
\ACCEPTED{May 2, 2019} 

\ARXIVID{1807.02354} 

\VOLUME{24}
\YEAR{2019}
\PAPERNUM{57}
\DOI{https://doi.org/10.1214/19-EJP314}

\ABSTRACT{%
	We study a one-dimensional spatial population model where the population sizes of the subpopulations are chosen according to a translation invariant and ergodic distribution and are uniformly bounded away from 0 and infinity.
	We suppose that the frequencies of a particular genetic type in the colonies evolve according to a system of interacting diffusions, following the stepping stone model of Kimura.
	We show that, over large spatial and temporal scales, this model behaves like the solution to a stochastic heat equation with Wright-Fisher noise with constant coefficients.
	These coefficients are the effective diffusion rate of genes within the population and the effective local population density.
	We find that, in our model, the local heterogeneity leads to a slower effective diffusion rate and a larger effective population density than in a uniform population.
	Our proof relies on duality techniques, an invariance principle for reversible random walks in a random environment and a convergence result for a system of coalescing random walks in a random environment.}

\DeclareDocumentCommand\dual{m g}{
	\left\langle#1,\IfNoValueTF{#2}{\phi}{#2}\right\rangle}

\DeclareDocumentCommand{\mean}{m}{\int #1 d \mu}

\DeclareDocumentCommand\Pq{o m g}{
	\mathbb{P}^\omega%
	\IfNoValueF{#1}{_{#1}}
	\left(%
	\IfNoValueF{#3}{\left.}%
	#2%
	\IfNoValueF{#3}{\:\right|\: #3}
	\right)}
\DeclareDocumentCommand\Eq{o m g}{
	\mathbb{E}^\omega%
	\IfNoValueF{#1}{_{#1}}
	\left[%
	\IfNoValueF{#3}{\left.}%
	#2%
	\IfNoValueF{#3}{\:\right|\: #3}
	\right]}

\newcommand{\Nm}{N_{3}}

\newcommand{\etalchar}[1]{$^{#1}$}
\providecommand{\bysame}{\leavevmode\hbox to3em{\hrulefill}\thinspace}
\providecommand{\MR}{\relax\ifhmode\unskip\space\fi MR }

\providecommand{\href}[2]{#2}

\begin{document}



\section*{Introduction}

Most species occupy a spatially extended habitat, where each individual produces some quantity of offspring which disperse around them.
If this dispersion is limited compared to the whole range of the population, spatial patterns of genetic diversity can build up over many generations, and individuals living far apart tend to be more genetically different than individuals living close to one another.

If sufficiently many individuals from different geographic locations are sampled and genotyped, one can estimate the parameters governing the dispersal and reproduction of individuals within the studied population \cite{rousset_genetic_1997,barton_inference_2013,ringbauer_inferring_2017}.
In particular, one is interested in the diffusion coefficient of genes inside the population, corresponding to the mean squared distance between parents and their offspring, and in the effective population density, which is the mean density of reproducing individuals per generation \cite{barton_neutral_2002}.
This estimation relies mainly on the comparison of the observed geographical decrease of genetic relatedness to theoretical models of isolation by distance \cite{wright_isolation_1943,malecot_les_1948,kimura_stepping_1964}.

To derive analytical formulae, these models often assume that the population density is uniform in space and constant in time \cite{sawyer_results_1976,sawyer_asymptotic_1977}.
Natural populations, however, are hardly likely to satisfy this assumption, either because of external conditions (\textit{e.g.} climate variations or fluctuating resources) or because of demographic fluctuations.
These heterogeneities are almost certain to affect the estimates of the demographic parameters, either increasing or decreasing the effective rate of diffusion and the effective population density.

In this paper, we investigate the effect of local heterogeneities in the population density on the large scale genetic composition of the population and on the corresponding demographic parameters.
We assume that the population density varies in space according to a translation invariant distribution and that it is constant in time.
In our model, we find that these heterogeneities lead to a slower diffusion of genes and a larger effective population density than in a population with uniform density.
According to this, significant local heterogeneities could lead to an underestimation of the migration levels and an overestimation of the population densities in natural populations.

This result relies on a study of the stepping stone model introduced by Kimura in 1953 \cite{kimura_stepping-stone_1953}.
In this model, the population is divided into subpopulations, or demes, sitting at the vertices of a graph (\textit{e.g.} $ \Z $ or $ \Z^2 $), and in each generation, new individuals are chosen at random to form the new generation in each deme, some of which are issued from parents within the same deme while others descend from individuals in neighbouring demes.
If the number of individuals in each deme is large enough and if the migration probabilities are small enough, the evolution of the frequency of a given genetic type can be described by a system of interacting diffusions solving
\begin{align} \label{stepping_stone_intro}
d p_t(x) = \sum_{y \sim x} m_{xy} (p_t(y) - p_t(x)) dt + \sqrt{\frac{1}{N(x)} p_t(x) (1-p_t(x))} d B^x_t,
\end{align}
where $ p_t(x) $ is the proportion of individuals carrying the type in deme $ x $ at time $ t $, $ m_{xy} $ is the probability that an individual in deme $ x $ is issued from one living in deme $ y $ in the previous generation, $ N(x) $ is number of individuals in deme $ x $ and $ ( B^x_t, t\geq 0, x \in \Z) $ is a family of independent Brownian motions \cite{etheridge_mathematical_2011,shiga_stepping_1988}.

Kimura and Weiss \cite{kimura_stepping_1964}, followed by Sawyer \cite{sawyer_results_1976} showed how the correlations in the genetic composition of demes decrease with the distance between them.
In particular, in a one-dimensional space, they showed that the correlation coefficient between the genetic composition of two demes separated by a distance $ r $ decreases like $ e^{- \lambda r} $, where $ \lambda > 0 $ depends on the parameters of the model.

This model can also be extended to a (one-dimensional) continuous geographical space in the following way \cite{shiga_stepping_1988,donnelly_continuum-sites_2000}.
For $ n \geq 1 $ and $ x \in \frac{1}{\sqrt{n}}\Z $, set
\begin{align} \label{rescaling_intro}
\bm{p}_n(t,x) = p_{nt}(\sqrt{n}x),
\end{align}
and assume $ m_{xy} = \frac{1}{2}m $ if $ \abs{x-y} = 1 $, $ (1-m) $ if $ x=y $ and $ 0 $ otherwise and that $ N(x) = \sqrt{n}/\gamma $ for all $ x \in \Z $ for some $ \gamma > 0 $.
Then, as $ n \to \infty $, the sequence of processes $ (\bm{p}_n(t,\cdot), t \in [0,T]) $ converges in distribution to $ (\bm{p}_t, t \in [0,T]) $, weak solution of
\begin{align} \label{SHE_intro}
\partial_t \bm{p}_t = \frac{\sigma^2}{2} \partial_{xx} \bm{p}_t + \sqrt{\gamma \bm{p}_t(1-\bm{p}_t)} \dot{W},
\end{align}
where $ \dot{W} $ is space-time white noise and $ \sigma^2 = m $.
A related convergence result was proved for the long range voter model in \cite{mueller_stochastic_1995}, and for the \slfv in \cite{etheridge_rescaling_2018}.
A corresponding convergence result regarding the underlying genealogies was also proved in \cite{greven_continuum_2016}.

To study the impact of local heterogeneities in the population size, we consider a stepping stone model where the population sizes are drawn at random from a translation invariant and ergodic distribution and remain fixed in time.
We then suppose that, with probability $ 1-m $, individuals in deme $ x $ descend from an individual in the same deme, and with probability $ m $, they descend from an individual chosen uniformly at random from the individuals living in demes $ \{x-1, x, x+1\} $.
The probability of having an ancestor in a neighbouring deme is thus proportional to the population size in this deme.
More precisely, an individual in deme $ x $ has its parent in deme $ y $ with probability
\begin{equation} \label{migration_probas}
\begin{aligned}
0 \hspace{6.2cm} & \text{ if } \abs{x-y} > 1, \\
m \frac{N(y)}{N(x-1) + N(x) + N(x+1)} \hspace{1.1cm} & \text{ if } \abs{x-y} = 1 \\
1 - m\frac{N(x-1) + N(x+1)}{N(x-1) + N(x) + N(x-1)} \hspace{0.5cm} & \text{ if } x=y.
\end{aligned}
\end{equation}
We then define (see Definition~\ref{def:stepping_stone_re} below) a system of interacting diffusions in a random environment analogous to \eqref{stepping_stone_intro} in this setting and we show that, if the population sizes are uniformly bounded away from 0 and infinity, the solution to this system can be rescaled as in \eqref{rescaling_intro} and converges as $ n \to \infty $ to a similar limit as in the uniform setting, \textit{i.e.} \eqref{SHE_intro}.
The two parameters $ \sigma^2 $ and $ \gamma $ of the limiting equation are expressed as functions of the distribution of the population sizes.
In particular, we find that (see Remark~\ref{remark:parameters_inequalities} below), with this particular choice of migration probabilities and under some technical assumptions,
\begin{align*}
\sigma^2 \leq \frac{2}{3} m, && \gamma \leq \frac{1}{\langle N \rangle},
\end{align*}
where $ \langle N \rangle $ denotes the average deme size and $ \frac{2}{3} m $ would be the expected diffusion coefficient if the population sizes were uniform.
As a result, local demographic heterogeneities appear to reduce the effective diffusion of genes within the population and to increase the effective population size.
This is due to the fact that, in our model, the ancestors of current individuals are more likely to have lived in more crowded demes, from which the diffusion is slower and in which the perceived population size is larger.

This bias should in particular be taken into account when estimating levels of gene flow and effective population sizes from large scale genetic patterns.
One should be careful, however, that these biases are a direct consequence of the choice made in \eqref{migration_probas}, and that different migration patterns might lead to different results (see Remark~\ref{remark:different_migration} below).

The proof of the main result relies on a duality relation between the stepping stone model \eqref{stepping_stone_intro} and a system of coalescing random walks which describes the genealogy of a random sample of individuals in the population \cite{etheridge_mathematical_2011,shiga_stepping_1988}.
In our model, this dual takes the form of a system of coalesing random walks in a random environment, whose jump and coalescence rates are both affected by the environment (the random walkers - or ancestral lineages - are more likely to jump to more crowded demes and coalesce more quickly if they meet in less crowded demes).

It turns out that these random walks admit a reversible measure on $ \Z $. This allows us to prove an invariance principle for the trajectory of each of these random walks \cite{kozlov_method_1985,de_masi_invariance_1989,lam_quenched_2014,derrien_local_2015}.
To characterize the coalescence time of a pair of lineages, we introduce an auxiliary process which records the sequence of demes where the two lineages meet, and we show that this process also admits a reversible measure (which happens to be the square of the previous one, see Proposition~\ref{prop:ergodicity_environment_two_particles}).
We thus show that the rescaled dual process converges to a system of Brownian motions which coalesce at a rate proportional to their local time together.
In turn, this system is known to be dual to a weak solution to equation \eqref{SHE_intro}, which allows us to uniquely characterise the possible limits of the rescaled stepping stone model.
We then conclude the proof of the main result by proving the tightness of this sequence.

Related models have been studied before, for example \cite{greven_interacting_2001} introduce a system of interacting Wright-Fisher diffusions in a space-time random environment and study its long time behaviour.
Moreover, in \cite{birkner_directed_2013}, the authors introduce a random walk on the
backbone of an oriented percolation cluster.
Given a realisation of a discrete-time contact process on $ \Z^d $, $ (\eta_n, n \in \Z) $, taking values in $ \lbrace 0, 1 \rbrace^{\Z^d} $, they consider a random walk describing the position of the ancestor of a particle at time 0.
This random walk $ (X_n, n \geq 0) $ satisfies
\begin{align*}
\P{X_{n+1} = y}{X_n = x, (\eta_n, n \in \Z)} \propto \eta_{-n-1}(y) \1{\| x - y \| = 1}.
\end{align*}
The authors then prove a law of large numbers and a quenched and annealed central limit theorem for this random walk in a dynamical random environment.
Their proof also covers the case were one puts (random) weights on the available sites in the contact process by replacing $ (\eta_n(x), x \in \Z^d, n \in \Z) $ with $ (\eta_n(x)K(x,n), x \in \Z^d, n \in \Z) $ with $ (K(x,n), x \in \Z^d, n \in \Z) $ \textit{i.i.d.} $ \N^* $-valued random variables.
The random field $ K $ can then be seen as a carrying capacity which fluctuates in space and time, superimposed on the contact process $ \eta $ which determines the availability of sites.
This was then generalised by K. Miller in \cite{miller_random_2016}, where $ K $ can be an arbitrary mixing field taking values in $ (0,\infty) $.
Finally, Birkner et al. \cite{birkner_random_2016} extended this to more general settings where $ \eta $ is a stationary Markovian particle system whose evolution can be described by ‘local rules.’

Our setting is simpler than in these works in that we consider an environment (the deme sizes) which is fixed in time, and uniformly elliptic (\textit{i.e.} bounded away from zero and infinity), but we are interested in the asymptotic behaviour of more than one ancestral lineage in the population.
For this we have to study the coalescence of pairs of random walks in a random environment.
This was done in a particular setting in \cite{birkner_coalescing_2018}, where the authors consider a family of random walks following the model in \cite{birkner_directed_2013} which coalesce whenever they find themselves in the same site.
They show that, adequately rescaled, the resulting process converges in distribution to the Brownian web (\textit{i.e.} a family of Brownian motions which coalesce instantly upon meeting, see for example \cite{fontes_brownian_2004} or \cite{schertzer_brownian_2015}).
Here, however, we explore the regime of so-called \textit{delayed coalescence}, where lineages must spend a positive amount of time together before they can coalesce (see Theorem~\ref{thm:cvg_brownian_flow}).
This adds an additional dependence on the law of the environment in the limiting process since the rate of coalescence depends on the local population density.

Another related work is \cite{greven_hierarchical_2018}, where a large population is divided in colonies labelled by the hierarchical group of order $ N $, and migration and reproduction can take place within each hierarchical block, and where the rates of these events for each block is random.

The paper is laid out as follows.
In the first section, we define the stepping stone model in a random environment and state our main result, namely the convergence of the sequence of rescaled stepping stone models to a weak solution to equation \eqref{SHE_intro}.
In Section~\ref{sec:RWRE}, we define the dual of the stepping stone model in a random environment and we state the related convergence results, first for individual lineages and then for the whole dual process.
These results concerning the dual are then proved in Section~\ref{sec:rescaling_dual}.
In Section~\ref{sec:proof_cvg_SHE}, we prove our main result, using estimates on the heat kernel associated to the random environment.
These estimates are proved in Appendix~\ref{appendix:heat_kernel_estimates}.
Appendix~\ref{sec:local_clt} is then devoted to the proof of a local central limit theorem for the random walk in a random environment which appears in the dual process.

\section*{Acknowledgements}

The author wishes to thank Nina Gantert for helpful discussions at an early stage of this project, and for bringing the work of \cite{derrien_local_2015} to his attention.
The author is also indebted to Alison Etheridge and Amandine Véber for their advice all along this project.
Finally, the author would like to thank an anonymous referee and an associate editor who carefully reviewed this paper and significantly helped to improve its presentation.

\section{The stepping stone model in a random environment}  \label{sec:stepping_stone_re}

\subsection{Definition of the model} \label{subsec:model}

We define a model for the evolution of a population living in discrete colonies or demes located on the one-dimensional integer lattice $ \Z $.
Moreover, we wish to draw the deme sizes (\textit{i.e.} the effective number of individuals living in each colony) at random from some translation invariant, ergodic distribution.

For $ x \in \Z $, let $ T^x $ be the translation operator acting on functions from $ \Z $ to $ \R $ defined by
\begin{align*}
T^xf(z) = f(x+z), \quad \forall z \in \Z.
\end{align*}
Thus let $ N = \{ N(x), x \in \Z \} $ be an $ \R^\Z $-valued random variable such that
\begin{enumerate}[i)]
	\item (translation invariance) for all $ x \in \Z $, $ T^x N \overset{d}{=} N $, where $ \overset{d}{=} $ stands for equality in distribution,
	\item (ergodicity) for all $ f : \R^\Z \to \R $ such that $ T^x f = f $ for all $ x \in \Z $, $ f(N) $ is deterministic (or almost surely constant),
	\item (uniform ellipticity) there exists $ K > 0 $ such that, almost surely, for all $ x \in \Z $, 
	\begin{align} \tag{$\mathcal{U.E.}$} \label{uniform_ellipticity}
	\frac{1}{K} \leq N(x) \leq K.
	\end{align}
\end{enumerate}
In words, we assume that the distribution of the population sizes is invariant by translation, that any translation invariant statistic is deterministic and that they are uniformly bounded away from 0 and infinity, almost surely.

Such a random variable can be defined using the usual formalism of random walks in random environments in the following way. Let $ (\Omega, \mathcal{B}, \mu) $ be a probability space on which is defined a family of measurable maps $ (T^x)_{x\in \Z} $, $ T^x : \Omega \to \Omega $ such that
\begin{enumerate}[a)]
	\item $ T^x \circ T^y = T^{x+y} $ for all $ x, y \in \Z $ and $ T^0 = Id_{\Omega} $,
	\item $ T^x $ is measure-preserving for all $ x \in \Z $, \textit{i.e.}
	\begin{align*}
	\mu(T^x A) = \mu(A) \quad \forall A \in \mathcal{B},
	\end{align*}
	\item the family $ (T^x)_{x\in\Z} $ is ergodic \wrt the measure $ \mu $, \textit{i.e.}
	\begin{align*}
	T^x A = A \quad \forall x \in \Z \implies \mu(A) \in \{0, 1 \}.
	\end{align*}
\end{enumerate}
Now let $ N : \Omega \to \R_+ $ be a random variable on $ \Omega $ such that there exists $ K >0 $ with
\begin{align*} 
\frac{1}{K} \leq N(\omega) \leq K \qquad \mu(d \omega) \text{-almost surely}.
\end{align*}
For $ x \in \Z $, we set
\begin{align*}
N(x) = N(\omega, x) = N(T^x \omega).
\end{align*}
Then $ \{ N(x), x \in \Z \} $ satisfies assumptions \textit{(i)-(iii)} above.
The parameter $ \omega \in \Omega $ determines the environment in which the population evolves, and the whole process will be defined on a larger (unspecified) probability space.
All our results will be \textit{quenched}, \textit{i.e.} they will hold conditionally on $ \omega $, for $ \mu $-almost every $ \omega $ in $ \Omega $.

\begin{remark}
	Note that $ N(x) $, $ N(\omega, x) $ and $ N(T^x \omega) $ all mean the same, and we sometimes change notations according to what is more practical.
	Most of the time we use $ N(\omega, x) $ if we want to make the dependence on $ \omega $ explicit, and $ N(x) $ if it can be omitted.
	The same applies to other random variables defined on the probability space $ (\Omega, \mathcal{B}, \mu) $.
\end{remark}

Conditionally on the deme sizes $ \{ N(x) , x \in \Z \} $, we define the stepping stone model in a random environment as follows.
Let $ \F (\Z, [0,1] ) $ denote the space of functions from $ \Z $ to $ [0,1] $.
For $ x \in \Z $, set
\begin{align*}
\Nm(x) = \Nm(\omega, x) = N(\omega, x-1) + N(\omega, x) + N(\omega, x+1).
\end{align*}

\begin{definition}[The stepping stone model in a random environment] \label{def:stepping_stone_re}
	Let $ (\Omega, \mathcal{B}, \mu) $, $ (T^x)_{x \in \Z} $ and $ N $ be as above.
	Fix $ p^0 \in \F(\Z, [0,1]) $, $ \lambda > 0 $ and $ m > 0 $.
	The stepping stone model in a random environment is defined as the $ \F(\Z, [0,1]) $-valued process $ (p_t(\omega,\cdot), t\geq 0) $ solving
	\begin{braceqn} \label{stepping_stone_equation}
		& d p_t(\omega, x) = m \sum_{z \in \{-1,1\}} \frac{N(\omega,x+z)}{\Nm(\omega,x)} \left( p_t(\omega, x+z) - p_t(\omega,x) \right) dt \\
		& \hspace{140pt} + \sqrt{\frac{1}{\lambda N(\omega, x)} p_t(\omega, x) (1-p_t(\omega,x)) } d B^x_t, \\
		& p_0(\omega,x) = p^0(x),
	\end{braceqn}
	for $ \mu $-almost every $ \omega $, where $ (B^x, x\in \Z) $ is a family of independent standard \Bm* (which is also independent from $ N $).
\end{definition}

In other words, we choose the population size of deme $ x $ to be $ \lambda N(x) $ and the probability that an individual in deme $ x $ at time $ t $ has a parent in deme $ x + z $ at time $ t - dt $ for $ z \in \{ -1, 1 \} $ is
\begin{align*}
m \frac{N(x + z)}{N_3(x)} dt + \littleO{dt}.
\end{align*}
This corresponds to the large population limit of an interacting Moran model where, with probability $ m $, the parent of an individual in deme $ x $ is drawn uniformly at random from the three populations $ \{ x-1, x, x+1 \} $, and with probability $ 1-m $, it is drawn uniformly from deme $ x $, see \cite{greven_representation_2005}.

\begin{remark}
	The fact that the process $ (p_t(\omega,\cdot), t \geq 0) $ is uniquely defined results from \cite{shiga_stepping_1988} (Section~2) where existence and uniqueness of the stepping stone model is proved in a fixed but arbitrary environment satisfying \eqref{uniform_ellipticity}.
	Furthermore we note that for each $ x \in \Z $, $ t \mapsto p_t(\omega,x) $ is almost surely continuous.
\end{remark}

We denote the quenched distribution of $ (p_t(\omega, \cdot), t\geq 0) $ with initial condition $ p^0 $ by $ \mathbb{P}_{p^0}^\omega $ (\textit{i.e.} the distribution of the process conditionally on the environment $ \omega $).
Expectation \wrt the quenched distribution will be denoted by $ \mathbb{E}_{p^0}^\omega $.

\subsection{Main result - rescaling limit of the stepping stone model in a random environment} \label{subsec:main_result}

Individuals are related when they share at least one common ancestor some time in the past.
If two individuals are sampled at a distance $ \sqrt{n} $ of each other, they need to look back at least $ n $ generations in the past to expect to have ancestors living in the same deme.
Over time scales of the order of $ n $ generations, individuals sampled from the same deme will have ancestors living in the same deme around $ \sqrt{n} $ generations.
Each time they do, they have a probability $ 1/N $ of having a common genealogical ancestor in the previous generation.
Hence if $ N $ is of the order of $ \sqrt{n} $, a positive proportion of individuals living in these demes should be related after a time of the order of $ n $ generations.

For this reason, we shall study the behaviour of $ (p_t(\omega,\cdot), t \geq 0) $ on spatial scales of the order of $ \sqrt{n} $ and at times of the order of $ n $ with $ \lambda = \sqrt{n} $ as $ n $ tends to infinity, for a fixed environment $ \omega $.
When the deme sizes are uniform, it is well known that the stepping stone model rescaled in this way converges in distribution to a weak solution of the stochastic heat equation with Wright-Fisher noise \eqref{SHE_intro} \cite{mueller_stochastic_1995,shiga_stepping_1988}.
Our main result states that on these spatial and temporal scales, the process forgets the details of the environment and evolves as an effective population with uniform demographic parameters.

For $ n \geq 1 $, let $ p^0_n \in \F(\Z, [0,1]) $ and set $ \lambda_n = \sqrt{n} $.
Let $ (p^{n}_t(\omega,\cdot), t\geq 0) $ be the solution of \eqref{stepping_stone_equation} with initial condition $ p^0_n $ and $ \lambda = \lambda_n $, and set, for $ n \geq 1 $,
\begin{align*}
\bm{p}^n_t(\omega,x) = p_{nt}^{n}(\omega, \sqrt{n} x), \quad x \in \frac{1}{\sqrt{n}} \Z.
\end{align*}
It is convenient to view $ (\bm{p}^n_t(\omega,\cdot), t \geq 0) $ as a process taking values in the space $ \Xi $ of Radon measures on $ \R $ through the identification
\begin{align*}
\bm{p}^n_t(\omega, dx) = \sum_{y \in \Z} p^{n}_{nt}(\omega, y) \frac{1}{\sqrt{n}} \delta_{y/\sqrt{n}}(dx).
\end{align*}

Let $ \mathcal{C}^\infty_c(\R) $ be the space of smooth and compactly supported real-valued functions on $ \R $.
For $ p \in \Xi $ and $ \phi \in \mathcal{C}^\infty_c(\R) $, set
\begin{align*}
\dual{p} = \int_{\R} \phi(x) p(dx).
\end{align*}
In this way, for any compactly supported function $ \phi : \R \to \R $,
\begin{align} \label{integral_pn}
\langle \bm{p}^n_t, \phi \rangle = \frac{1}{\sqrt{n}} \sum_{x \in \Z} p^n_{nt}(\omega,x) \phi \left( \dfrac{x}{\sqrt{n}} \right).
\end{align}
The space $ \Xi $ is endowed with the topology of vague convergence (a sequence of measures $ ( \nu_n )_n $ is said to converge \textit{vaguely} to $ \nu \in \Xi $ if $ \langle \nu_n, \phi \rangle \to \langle \nu, \phi \rangle $ for all $ \phi \in C^\infty_c(\R) $).
Let $ \seq{\phi} $ be a uniformly bounded separating family of $ \mathcal{C}^\infty
_c(\R) $.
Then
\begin{align*}
d(p,q) = \sum_{n \geq 1} \frac{1}{2^n} \abs{\dual{p}{\phi_n} - \dual{q}{\phi_n}}
\end{align*}
defines a metric for the vague topology on $ \Xi $.

Also let $ \mathcal{C}([0,T], \Xi) $ denote the space of continuous functions from $ [0,T] $ to $ (\Xi, d) $, endowed with the uniform topology.
The main result of this paper is the following.

\begin{theorem}[Convergence to the stochastic heat equation with Wright-Fisher noise] \label{thm:cvg-SHE-WF-noise}
	Fix $ T > 0 $ and suppose that $ (\Omega, \mathcal{B}, \mu) $, $ (T^x)_{x \in \Z} $, $ N $ satisfy \textit{(a-b-c)} and \eqref{uniform_ellipticity}.
	Assume that $ \bm{p}^n_0 $ is deterministic and converges vaguely to $ \bm{p}_0 \in \Xi $ and that $ \bm{p}_0 $ is absolutely continuous \wrt the Lebesgue measure.
	Further assume that $ p^n_0 $ satisfies the following uniform H\"older estimate:
	\begin{align} \label{holder_initial_condition}
	\sup_{n \geq 1} \sup_{x, y \in \Z} \sqrt{n} \frac{\abs{p^n_0(x) - p^n_0(y)}}{\abs{x-y}} < + \infty.
	\end{align}
	Then, $ \mu(d\omega) $-almost surely, as $ n \to \infty $, the sequence of $ \Xi $-valued processes $ (\bm{p}^n_t(\omega, \cdot), t \in [0,T]) $ converges in distribution in $ \mathcal{C}([0,T], \Xi) $ to a $ \Xi $-valued process $ (\bm{p}_t, t \in [0,T]) $ such that $ \bm{p}_t $ is absolutely continuous with respect to the Lebesgue measure for every $ t \geq 0 $ almost surely and such that, for any $ \phi \in \mathcal{C}^\infty_c(\R) $,
	\begin{align} \label{martingale_pb_SHE}
	\langle \bm{p}_t, \phi \rangle - \langle \bm{p}_0, \phi \rangle - \frac{\sigma^2}{2} \int_{0}^{t} \langle \bm{p}_s, \partial_{xx} \phi \rangle ds
	\end{align}
	is a continuous local martingale with quadratic variation
	\begin{align*}
	\gamma \int_{0}^{t} \langle \bm{p}_s (1-\bm{p}_s), \phi^2 \rangle ds,
	\end{align*}
	where this term is well defined since $ \bm{p}_s $ is absolutely continuous and where $ \sigma^2 $ and $ \gamma $ are given by
	\begin{align} \label{parameters}
	\sigma^2 = 2m \left( \mean{\frac{1}{N T^1 N}} \mean{N \Nm} \right)^{-1}, && \gamma = \frac{\mean{N (\Nm)^2}}{\mean{(N \Nm)^2}}.
	\end{align}
\end{theorem}

Theorem~\ref{thm:cvg-SHE-WF-noise} states that, over large spatial and temporal scales, the stepping stone model in a random environment behaves as if it were in a homogeneous environment with effective parameters $ \sigma^2 $ and $ \gamma $.
Furthermore, the explicit formulas in \eqref{parameters} yield the following inequalities.

\begin{remark} \label{remark:parameters_inequalities}
	\begin{enumerate}
		\item We have
		\begin{align*}
		\sigma^2 \leq \frac{2}{3} m,
		\end{align*}
		where $ \frac{2}{3}m $ would be the expected diffusion coefficient if $ N $ were constant.
		To see this, write, by Jensen's inequality and the Cauchy-Schwartz inequality,
		\begin{align*}
		\mean{\frac{1}{N T^1N}}\mean{N \Nm} \geq \frac{2 \mean{N T^1 N} + \mean{N^2}}{\mean{N T^1 N}} \geq 3.
		\end{align*}
		The local variations in the population density are thus seen to reduce the effective dispersion of genes in the population.
		\item In addition, if there exists a non-decreasing measurable function $ g : \R_+ \to \R_+ $ such that $ \E{N N_3^2}{N} = g(N) $ almost surely (which is the case for example if $ N $, $ T^1 N $ and $ T^{-1} N $ are independent), then Lemma~\ref{lemma:coupling} below implies
		\begin{align*}
		\gamma \leq \frac{1}{\mean{N}}.
		\end{align*}
		To see this, write
		\begin{align*}
		\frac{1}{\gamma} = \int_\Omega N(\omega) \frac{N(\omega) \Nm(\omega)^2}{\mean{N \Nm^2}} \mu(d\omega)
		\end{align*}
		and apply Lemma~\ref{lemma:coupling} with $ X = N $ and $ Y = N N_3^2 $.
		This shows that the effective population density in a heterogeneous population is larger than the mean population density.
		We shall see below that this is due to the fact that individuals are more likely to descend from more crowded regions, where the coalescence rate is lower and the perceived population density is larger.
	\end{enumerate}
\end{remark}

\begin{lemma} \label{lemma:coupling}
	Let $ X $ and $ Y $ be two $ L^1 $ random variables taking values in $ \R_+ $.
	Suppose that there exists a \textup{non-decreasing} measurable function $ g : \R_+ \to \R_+ $ such that $ \E{Y}{X} = g(X) $ almost surely.
	Then
	\begin{align*}
	\E{ X Y } \geq \E{X} \E{Y}.
	\end{align*}
\end{lemma}

Lemma~\ref{lemma:coupling} (which can be seen as a generalisation of a classical inequality for size-biased sampling) is proved in Appendix~\ref{appendix:coupling}.
Theorem~\ref{thm:cvg-SHE-WF-noise} is proved in Section~\ref{sec:proof_cvg_SHE} where we show that the sequence $ (\bm{p}^n_t(\omega, \cdot), t\geq 0) $ is tight in $ \mathcal{C}([0,T], \Xi) $ $ \mu(d\omega) $-almost surely and we identify the limit through a duality relation.

\begin{remark} \label{remark:different_migration}
	Alternative migration mechanisms can readily be considered instead of the one chosen in Definition~\ref{def:stepping_stone_re}.
	For example it also seems natural to assume that each individual tries to send $ \frac{m}{2} $ offspring to each neighbouring deme and $ (1-m) $ offspring to its current deme.
	Assuming exactly this would violate the assumption that the deme sizes are constant (unless they are all equal), but we can approach this situation for small values of $ m $ as follows.
	
	Suppose that an individual in deme $ x \pm 1 $ sends
	\begin{align*}
	\frac{\frac{m}{2} N(x)}{N(x-1)\frac{m}{2} + N(x+1)\frac{m}{2} + N(x)(1-m)}
	\end{align*}
	offspring to deme $ x $ and that an individual in deme $ x $ sends
	\begin{align*}
	\frac{N(x)(1-m)}{N(x-1)\frac{m}{2} + N(x+1)\frac{m}{2} + N(x)(1-m)}
	\end{align*}
	offspring to deme $ x $.
	Then the probability that an individual in deme $ x $ has its parent in deme $ y $ is
	\begin{equation*}
	\begin{aligned}
	0 \hspace{2.7cm} & \text{ if } \abs{x-y} > 1, \\
	\frac{m}{2} \frac{N(y)}{N_3(x)} \hspace{1.4cm} & \text{ if } \abs{x-y} = 1, \\
	(1 - m) \frac{N(x)}{N_3(x)} \hspace{0.5cm} & \text{ if } x=y,
	\end{aligned}
	\end{equation*}
	where we have set
	\begin{align} \label{alt_def_N_3}
	N_3(x) = N(x-1)\frac{m}{2} + N(x+1)\frac{m}{2} + N(x)(1-m).
	\end{align}
	
	It can then be verified that Theorem~\ref{thm:cvg-SHE-WF-noise} holds without modification if we replace $ m $ by $ \frac{m}{2} $ and we take \eqref{alt_def_N_3} as the definition of $ N_3(x) $.
	In particular, \eqref{parameters} still holds and we can see as in Remark~\ref{remark:parameters_inequalities} that
	\begin{align*}
	\sigma^2 \leq m,
	\end{align*}
	where $ m $ would be the expected diffusion coefficient if the deme sizes were all equal.
	Similarly, if there exists a non-decreasing measurable function $ g : \R_+ \to \R_+ $ such that $ \E{N N_3^2}{N} = g(N) $ almost surely, then
	\begin{align*}
	\gamma \leq \frac{1}{\mean{N}}.
	\end{align*}
	Hence in this setting the local heterogeneities have the same qualitative effect on the effective diffusion coefficient and the effective population density as in the previous case.
\end{remark}

\section{Coalescing random walks in a random environment} \label{sec:RWRE}

It is well known that the stepping stone model of Definition~\ref{def:stepping_stone_re} admits a moment dual in the form of a system of coalescing random walks \cite{shiga_stepping_1988}.
These random walks describe the positions of the ancestors of a random sample of individuals in the population.
Each pair of random walks (also called ancestral lineages) coalesces at the first time in the past when the two sampled individuals have a common ancestor.
The lineages are affected by the heterogeneity of the environment in two ways: they are more likely to jump to more crowded demes, and they coalesce more quickly in less crowded demes.

In this section, we first define the system of coalescing random walks in a random environment which is dual to the stepping stone model of Definition~\ref{def:stepping_stone_re}.
We then state several results on this dual which will be used to identify the limit in the proof of Theorem~\ref{thm:cvg-SHE-WF-noise}.
First we state the convergence of the rescaled random walk to Brownian motion (Theorem~\ref{thm:functional_clt}) and a local central limit theorem for this random walk (Theorem~\ref{thm:local_clt}).
Then we state the convergence of the whole dual process to a system of coalescing Brownian motions (Theorem~\ref{thm:cvg_brownian_flow}).

\subsection{The dual of the stepping stone model in a random environment} \label{subsec:dual}

\begin{definition}[The dual of the stepping stone model in a random environment] \label{def:dual}
	Fix $ m > 0 $, $ \lambda > 0 $.
	For $ \omega \in \Omega $, $ k \geq 1 $ and $ \{ x_1, \ldots, x_k \} \in \Z^k $, let
	\begin{align*}
	\mathcal{A}^\omega_t = \{ \xi^1_t, \ldots, \xi^{N_t}_t \}, \quad t\geq 0,
	\end{align*}
	be a system of random walks on $ \Z $ such that
	\begin{enumerate}[i)]
		\item $ \mathcal{A}^\omega_0 = \{ x_1, \ldots, x_k \} $,
		\item each lineage currently in $ x \in \Z $ jumps to $ y \in \{x-1, x+1\} $ at rate
		\begin{align} \label{transition_rates}
		m \frac{N(\omega, y)}{\Nm(\omega,x)},
		\end{align}
		independently of the others,
		\item each pair of lineages sitting in the same colony $ x \in \Z $ coalesces at rate
		\begin{align*}
		\frac{1}{\lambda N(\omega,x)},
		\end{align*}
		independently of other pairs.
	\end{enumerate}
	For $ t \geq 0 $, we denote the number of lineages in $ \mathcal{A}^\omega_t $ by $ N_t $.
\end{definition}

Let $ \mathbb{P}_{\{ x_1, \ldots, x_k \}}^\omega $ denote the quenched distribution of $ \proc{\mathcal{A}^\omega} $ started from $ \{ x_1, \ldots, x_k \} $, and let $ \mathbb{E}_{\{ x_1, \ldots, x_k \}}^\omega $ denote the expectation \wrt this probability.
When $ k=1 $, $ \mathbb{P}_{\{ x \}}^\omega = \mathbb{P}_{x}^\omega $ is the quenched distribution of a random walk in a random environment with transition rates given by \eqref{transition_rates}.
In the following, we show that this random walk satisfies a quenched invariance principle (Theorem~\ref{thm:functional_clt} below).

When $ k = 2 $, we can give a more precise construction of $ \proc{\mathcal{A}^\omega} $.
Fix $ \{ x_1, x_2 \} \in \Z^2 $ and let $ \proc{\xi^1} $ and $ \proc{\xi^2} $ be two \emph{independent} random walks on $ \Z $ with transition rates given by \eqref{transition_rates} and started from $ x_1 $ and $ x_2 $, respectively.
Let $ E $ be an independent exponential random variable with parameter 1 and define, for $ t \geq 0 $,
\begin{align} \label{weighted_local_time}
L(t) = \int_{0}^{t} \frac{\1{\xi^1_s = \xi^2_s}}{\lambda N(\omega, \xi^1_s)} ds.
\end{align}
Now define the coalescence time of the two lineages as
\begin{align*}
T_c = \inf \{ t \geq 0 : L(t) > E \}.
\end{align*}
Setting
\begin{equation*}
\mathcal{A}^\omega_t = \left\lbrace
\begin{aligned}
\{ \xi^1_t, \xi^2_t \} & \text{ if } t < T_c, \\
\{ \xi^1_t \} & \text{ if } t \geq T_c,
\end{aligned}
\right.
\end{equation*}
we obtain a version of the process in Definition~\ref{def:dual} started from two lineages.
This construction can be extended to more than two lineages \cite{liang_two_2009}, but we will not need it in such generality here.

The next proposition states the duality relation between the stepping stone model and $ \proc{\mathcal{A}^\omega} $.

\begin{proposition}[Duality, \cite{shiga_stepping_1988}] \label{prop:duality}
	For any $ p^0 : \Z \to [0,1] $ and for any $ k \geq 1 $, $ \{ x_1, \ldots, x_k \} \in \Z^k $,
	\begin{align*}
	\Eq[p^0]{\prod_{i=1}^{k} p_t(\omega, x_i) } = \Eq[\{ x_1, \ldots, x_k \}]{ \prod_{i=1}^{N_t} p^0(\xi^i_t) } \quad \mu(d\omega)-a.s.
	\end{align*}
	where $ (p_t(\omega, \cdot), t\geq 0) $ is given by Definition~\ref{def:stepping_stone_re} and $ \proc{\A^\omega} $ by Definition~\ref{def:dual}.
\end{proposition}

We use this proposition to characterise the limiting behaviour of $ (\bm{p}^n_t, t \in [0,T]) $ as $ n \to \infty $ via the study of the rescaling limit of $ (\A^\omega_t, t \in [0,T]) $.
We show below that a rescaled version of $ (\A^\omega_t, t \in [0,T]) $ converges in distribution to a system of independent \Bm* which coalesce at a rate proportional to the local time at 0 of their difference (Theorem~\ref{thm:cvg_brownian_flow} below).

\subsection{The central limit theorem for reversible random walks in a random environment} \label{subsec:clt_RWRE}

Here, we state the results on the motion of a single lineage in the dual process.
For $ \omega \in \Omega $, let $ \proc{\xi} $ be a random walk on $ \Z $ with transition rates given by \eqref{transition_rates}, \textit{i.e.}, conditionally on the environment $ \omega $, it jumps from $ x \in \Z $ to $ y \in \{ x-1, x+1 \} $ at rate
\begin{align*}
m \frac{N(\omega,y)}{\Nm(\omega, x)}.
\end{align*}
The most notable property of this random walk is that it admits a reversible measure on $ \Z $, given by
\begin{align} \label{reversible_measure}
\pi(\omega, x) = \pi(T^x \omega), && \pi(\omega) = \frac{N(\omega) \Nm(\omega)}{\mean{N \Nm}}.
\end{align}
(Note that we have normalised $ \pi $ so that $ \pi(\omega) \mu(d \omega) $ is a probability measure on $ \Omega $.)
Together with \eqref{uniform_ellipticity}, this implies the central limit theorem for the random walk \cite{lam_quenched_2014,derrien_local_2015,depauw_variance_2009}, \textit{i.e.}
\begin{align*}
\frac{1}{\sqrt{t}} \xi_t \cvgas[d]{t} \mathcal{N}(0,\sigma^2) \quad \mu(d\omega)-a.s.
\end{align*}
with $ \sigma^2 $ as in \eqref{parameters}.
In Section~\ref{sec:rescaling_dual}, we prove that this extends to a quenched invariance principle for the random walk $ (\xi_t, t \geq 0) $.
For $ n \geq 1 $, set
\begin{align*}
\xi^n_t = \frac{1}{\sqrt{n}} \xi_{nt}.
\end{align*}
For $ T > 0 $, let $ \sko{\R} $ denote the space of \cadlag real-valued functions endowed with the usual Skorokhod topology.

\begin{theorem}[Functional central limit theorem for nearest neighbour reversible random walks] \label{thm:functional_clt}
	Fix $ T > 0 $ and assume that $ \xi^n_0 $ is deterministic and converges to $ x_0 \in \R $.
	Suppose that $ (\Omega, \mathcal{B}, \mu) $, $ (T^x)_{x \in \Z} $ and $ N $ satify (\textit{a-b-c}) and \eqref{uniform_ellipticity}.
	Then, for $ \mu $-almost every environment $ \omega $, as $ n \to \infty $, $ (\xi^n_t)_{t \in [0,T]} $ converges in distribution in $ \sko{\R} $ to \Bm started from $ x_0 $ with variance $ \sigma^2 $ given by~\eqref{parameters}.
\end{theorem}

The corresponding central limit theorem was proved for the random walk among random conductances in \cite{lam_quenched_2014} (see also \cite{lam_les_2012}).
Although our case doesn't take the form of a random walk among random conductances, the proof in \citep[Chapter~4]{lam_les_2012} only requires the reversibility of the random walk, and can be applied by replacing the conductivity with $ N(\omega)T^1N(\omega) $ and the reversible measure with $ N(\omega) \Nm(\omega) $.
We give the details in Section~\ref{sec:rescaling_dual}.

In \cite{derrien_local_2015}, Derrien proves a local central limit theorem for the random walk among random conductances which also extends to our setting.
We prove a slightly stronger version of his result, namely Theorem~\ref{thm:local_clt} below.
For $ \omega\in \Omega $, $ t \geq 0 $ and $ x, y \in \Z $, set
\begin{align} \label{def_g_omega}
g^\omega_t(x,y) = \Pq[x]{\xi_t = y}
\end{align}
and for $ t > 0 $, $ x \in \R $, define
\begin{align*}
G_t(x) = \frac{1}{\sqrt{2\pi \sigma^2 t}} \exp \left( - \frac{x^2}{2\sigma^2 t} \right).
\end{align*}

\begin{theorem}[Local central limit theorem for reversible random walks] \label{thm:local_clt}
	For all $ 0 < \varepsilon < T $ and $ R > 0 $ and for $ \mu $-almost every environment $ \omega $,
	\begin{align*}
	\lim_{n \to \infty} \: \sup_{t \in [\varepsilon, T]} \: \max_{\substack{x \in B(0,R\sqrt{n}) \cap \Z \\ y \in B(0, R \sqrt{n}) \cap \Z}} \: \abs{ \sqrt{n} \frac{g^\omega_{nt}(x,y)}{\pi(\omega,y)} - G_t \left( \frac{x-y}{\sqrt{n}} \right) } = 0.
	\end{align*}
\end{theorem}

In Section~\ref{sec:proof_cvg_SHE}, we shall use several estimates on the kernel $ g^\omega $ to show the tightness of the sequence $ ( \bm{p}^n_t(\cdot), t\geq 0) $.
These estimates are proved in Appendix~\ref{appendix:heat_kernel_estimates} and are then used to prove Theorem~\ref{thm:local_clt} in Appendix~\ref{sec:local_clt}.

\subsection{Delayed coalescence for random walks in a random environment} \label{subsec:delayed_coalescence}

Now that we know how each lineage behaves over large scales, we state the corresponding result for the dual process $ \proc{\A^\omega} $.
We limit ourselves to the dual started from two lineages, as this is enough to identify the limit of $ (\bm{p}^n_t(\omega,\cdot), t\geq 0) $ in the proof of Theorem~\ref{thm:cvg-SHE-WF-noise}.

Let us start with the definition of the limiting process.
It will be defined as a system of independent Brownian motions which coalesce at a rate proportional to their local time together.
For $ \sigma^2 > 0 $ and $ \gamma > 0 $, let $ \proc{X^1} $ and $ \proc{X^2} $ be two independent \Bm* with variance $ \sigma^2 $, started from $ x_1 $ and $ x_2 $, respectively.
Let $ E $ be an independent exponential random variable with parameter $ \gamma $, and let $ t \mapsto L^0_t(X^1-X^2) $ denote the local time at 0 of $ X^1 - X^2 $. Set
\begin{align*}
T_c = \inf \{ t \geq 0 : L^0_t(X^1-X^2) > E \}.
\end{align*}

\begin{definition}[Brownian flow with delayed coalescence, \cite{liang_two_2009}] \label{def:brownian_flow}
	The process $ \proc{\mathcal{D}} $ defined by
	\begin{equation*}
	\mathcal{D}_t = \left\lbrace
	\begin{aligned}
	\{ X^1_t, X^2_t \} & \text{ if } t < T_c, \\
	\{ X^1_t \} & \text{ if } t \geq T_c,
	\end{aligned}
	\right.
	\end{equation*}
	is called the Brownian flow with delayed coalescence with parameters $ ( \sigma^2, \gamma ) $.
\end{definition}

The name \textit{Brownian flow with delayed coalescence} is used in \cite{liang_two_2009} for a more general process started from an arbitrary number of lineages (which also coalesce at a rate proportional to their local time together). Here we only use the process started from two lineages.

\begin{remark}
	The process $ (\mathcal{D}_t, t \geq 0) $ takes values in the disjoint union of $ \R $ and $ \R^2 $, and we let $ \mathrm{D}([0,T],\R \cup \R^2) $ denote the space of \cadlag functions on $ [0,T] $ taking values in this space, endowed with the usual Skorokhod topology.
\end{remark}

The following result is proved in \cite{liang_two_2009} (see in particular Theorem~6.2 and Proposition~7.2).

\begin{proposition} \label{prop:duality_brownian_flow}
	There exists a unique Feller Markov process $ (\bm{p}_t)_{t \geq 0} $ with continuous sample paths taking values in the subspace
	$ \left\lbrace p \in \Xi : 0 \leq \langle p, \phi \rangle \leq \int_\R \phi(x) dx, \forall \phi \geq 0 \right\rbrace $
	such that, for $ k \in \lbrace 1, 2 \rbrace $, for any $ \phi_1, \phi_k $ in $ \mathcal{C}^\infty_c(\R) $,
	\begin{align} \label{duality_brownian_flow}
	\E[\bm{p}_0]{\prod_{i=1}^{k} \langle \bm{p}_t, \phi_i \rangle } = \int_{\R^k} \E[\lbrace x_1, x_k \rbrace]{\prod_{i=1}^{N_t} \bm{p}_0(X^i_t) } \phi_1(x_1)  \phi_k(x_k) dx_1 dx_k,
	\end{align}
	where $ \mathcal{D}_t = \{ X^1_t, X_t^{N_t} \} $ is the Brownian flow with delayed coalescence.
	In addition, for any twice continuously differentiable function $ f : \R \to \R $ and any $ \phi \in \mathcal{C}^\infty_c(\R) $,
	\begin{align*}
	f(\langle \bm{p}_t, \phi \rangle) - \int_{0}^{t} \left( \dfrac{\sigma^2}{2} \langle \bm{p}_s, \Delta \phi \rangle f'(\langle \bm{p}_s, \phi \rangle) + \frac{\gamma}{2} \langle \bm{p}_s(1-\bm{p}_s), \phi^2 \rangle f''(\langle \bm{p}_s, \phi \rangle) \right) ds
	\end{align*}
	is a local martingale with respect to the natural filtration of $ (\bm{p}_t)_{t \geq 0} $.
\end{proposition}

In \cite{liang_two_2009}, it is shown that there exists a unique Feller Markov process satisfying \eqref{duality_brownian_flow} for all $ k \geq 1 $, but since we ask that $ (\bm{p}_t, t \geq 0) $ has continuous sample paths, using Itô's formula, $ k \in \lbrace 1, 2 \rbrace $ is enough to identify the martingale problem satisfied by $ (\bm{p}_t, t \geq 0) $ on a suitable set of functions.

It follows that the process defined in Proposition~\ref{prop:duality_brownian_flow} is a weak solution to the stochastic heat equation with Wright-Fisher noise.

We now state our result on the scaling limit of the dual of the stepping stone model in a random environment.
For $ n \geq 1 $, take $ \lambda_n = \sqrt{n} $, $ x_1^n, x_2^n \in \Z $ and let $ ( \A^{\omega, \lambda_n}_t, t \geq 0) $ be the process of Definition~\ref{def:dual} with $ \lambda = \lambda_n $ started from $ \{ x_1^n, x_2^n \} $.
For $ n \geq 1 $, $ t \geq 0 $, set
\begin{align} \label{rescaled_dual}
\bm{\A}^n_t = \frac{1}{\sqrt{n}} \A^{\omega, \lambda_n}_{nt} = \left\lbrace \frac{1}{\sqrt{n}} \xi^1_{nt}, \frac{1}{\sqrt{n}} \xi^{N_{nt}}_{nt} \right\rbrace
\end{align}
(note that $ N_t \in \{1, 2\} $).

\begin{theorem}[Convergence to the Brownian flow with delayed coalescence] \label{thm:cvg_brownian_flow}
	Assume that $$ \frac{x^n_i}{\sqrt{n}} \cvgas{n} x_i $$ for $ i \in \{1, 2\} $.
	Then for $ \mu $-almost every environment $ \omega $, the process $ (\bm{\A}^n_t, t \in [0,T]) $ of \eqref{rescaled_dual} converges in distribution in $ \sko{\R \cup \R^2} $ to the Brownian flow with delayed coalescence $ (\mathcal{D}_t, t \in [0,T]) $ of Definition~\ref{def:brownian_flow} with $ \sigma^2 $ and $ \gamma $ given by \eqref{parameters}.
\end{theorem}

We prove Theorem~\ref{thm:cvg_brownian_flow} in Subsection~\ref{subsec:proof_cvg_brownian_flow}.
We already know from Theorem~\ref{thm:functional_clt} that each lineage converges in distribution to \Bm with variance $ \sigma^2 $.
It thus remains to show that the process $ L(t) $ defined in \eqref{weighted_local_time} becomes asymptotically proportional to $ L^0_t(X^1-X^2) $.
This is done by applying the ergodic theorem to an auxiliary process which is defined in Subsection~\ref{subsec:two_particles}.

Durrett and Restrepo \cite{durrett_one-dimensional_2008} already showed that, in the stepping stone model with \textit{uniform} population sizes and in the long range voter model, the time spent together by two lineages before they coalesce is asymptotically exponential (see also \cite{maruyama_rate_1971} and \cite{greven_continuum_2016}).
Theorem~\ref{thm:cvg_brownian_flow} extends this result to the stepping stone model in a random environment.

\begin{remark}
	The above result also holds if the process $ (\A^\omega_t, t \in [0,T]) $ is started from more than two lineages.
	The rescaled process $ (\bm{\A}^n_t, t \in [0,T]) $ then converges to the Brownian flow with delayed coalescence started from the corresponding number of lineages, as defined in \cite{liang_two_2009}.
\end{remark}

\section{Large scale behaviour of the dual of the stepping stone model in a random environment} \label{sec:rescaling_dual}

The aim of this section is to prove Theorem~\ref{thm:cvg_brownian_flow}.
We start in Subsection~\ref{subsec:proof_clt} by showing the quenched functional central limit theorem for the random walk $ (\xi_t)_{t\in [0,T]} $ (Theorem~\ref{thm:functional_clt}).
In Subsection~\ref{subsec:two_particles}, we introduce an auxiliary process called the environment viewed by the two random walks, and we show that it is ergodic and give its stationary distribution.
Theorem~\ref{thm:cvg_brownian_flow} is then proved in Subsection~\ref{subsec:proof_cvg_brownian_flow}.

\subsection{Quenched functional central limit theorem for the nearest neighbour reversible random walk} \label{subsec:proof_clt}

Here we prove Theorem~\ref{thm:functional_clt} using a martingale approximation of the random walk $ \proc{\xi^n} $ introduced in \cite{lam_les_2012}.
We start by defining a real-valued measurable function $ F(\omega,\cdot) : \Z \to \R $ by
\begin{equation} \label{definition_F}
F(\omega, k) = \left\lbrace
\begin{aligned}
\sum_{i=0}^{k-1} \frac{1}{N(\omega,i) N(\omega,i+1)} \hspace{25pt} & k \geq 1,\\
0 \hspace{125pt} & k=0, \\
- \sum_{i=1}^{-k} \frac{1}{N(\omega,-i) N(\omega,-i+1)} \hspace{10pt} & k \leq -1.
\end{aligned}
\right.
\end{equation}
We note that for all $ k, l \in \Z $ and $ \omega \in \Omega $,
\begin{align} \label{cocycle}
F(\omega, k+l) = F(\omega, k) + F(T^k\omega, l).
\end{align}
It follows (see \eqref{martingale_equation_F} below) that $ (F(\omega, \xi_{t}), t \geq 0) $ is a martingale.
By the pointwise ergodic theorem,
\begin{align} \label{scale_function}
\lim_{\abs{k} \to \infty} \frac{F(\omega, k)}{k} = \mean{\frac{1}{N T^1 N}}, \quad \mu(d\omega)-a.s.
\end{align}

We then decompose $ \xi^n_t = \frac{1}{\sqrt{n}} \xi_{nt} $ as
\begin{align*}
\xi^n_t = \left( \mean{\frac{1}{N T^1 N}} \right)^{-1} M^\omega_n(t) + R^\omega_n(t),
\end{align*}
with
\begin{align*}
M^\omega_n(t) = \frac{1}{\sqrt{n}} F(\omega, \xi_{nt}).
\end{align*}
Since, from \eqref{cocycle} and \eqref{definition_F},
\begin{align} \label{martingale_equation_F}
\frac{N(\omega,x+1)}{\Nm(\omega,x)} (F(\omega,x+1) - F(\omega,x)) + \frac{N(\omega,x-1)}{\Nm(\omega,x)} (F(\omega,x-1) - F(\omega,x)) = 0,
\end{align}
the process $ M^\omega_n $ is a martingale \wrt the natural filtration of $ \proc{\xi^n} $.

\begin{lemma} \label{lemma:remainder}
	For $ \mu $-almost every $ \omega $ and for any $ T >0 $,
	\begin{align*}
	\lim_{n \to \infty} \Eq[\xi^n_0]{ \sup_{t \in [0,T]} R^\omega_n(t)^2 } = 0,
	\end{align*}
	where the expectation is taken \wrt the distribution of $ \xi^n $ started at $ \xi^n_0 $ in the environment~$ \omega $ (recall that $ \xi^n_0 \to x_0 $).
\end{lemma}

Before proving Lemma~\ref{lemma:remainder}, let us show how it implies Theorem~\ref{thm:functional_clt}.
The only point left to be shown is that $ M^\omega_n $ converges in distribution to \Bm with the right variance.

\begin{proof}[Proof of Theorem~\ref{thm:functional_clt}]
	We begin by noting that the predictable quadratic variation of $ M^\omega_n $ is
	\begin{align} \label{qvar_Mn}
	\langle M^\omega_n \rangle_t = \frac{1}{n} \int_{0}^{nt} h(T^{\xi_s}\omega) ds 
	\end{align}
	with, after a simple calculation,
	\begin{align*}
	h(\omega) = \frac{m}{\Nm(\omega) N(\omega)^2} \left( \frac{1}{N(T^1\omega)} + \frac{1}{N(T^{-1}\omega)} \right).
	\end{align*}
	Let us define the so-called process of the environment viewed from the random walk as
	\begin{align} \label{def_Z}
	Z(t) = T^{\xi_t}\omega, \quad t \geq 0.
	\end{align}
	We wish to use the ergodic theorem to show that $ \langle M^\omega_n \rangle_t $ converges to a constant times $ t $.
	To do this, we show that the $ \Omega $-valued Markov process $ (Z(t), t \geq 0) $ with initial distribution $ \pi(\omega) \mu(d\omega) $ is stationary and ergodic (recall that $ \pi $ was defined in \eqref{reversible_measure}).
	Let $ \mathcal{L} $ denote the infinitesimal generator of $ (Z(t), t \geq 0) $.
	It is symmetric in $ L^2(\pi) $, \textit{i.e.}
	\begin{align*}
	\int_{\Omega} g(\omega) \mathcal{L} f(\omega) \pi(\omega) \mu(d\omega) = \int_{\Omega} \mathcal{L} g(\omega) f(\omega) \pi(\omega) \mu(d\omega)
	\end{align*}
	for all $ f, g \in L^2(\pi) $.
	Furthermore
	\begin{align*}
	\int_{\Omega} f(\omega) (-\mathcal{L})f(\omega) \pi(\omega) \mu(d\omega) = \frac{1}{2} \frac{m}{\mean{N \Nm}} \int_{\Omega} N(\omega) N(T^1\omega) \left( f(T^1\omega) - f(\omega) \right)^2 \mu(d\omega).
	\end{align*}
	Hence any function $ f $ in $ L^2(\pi) $ such that $ \mathcal{L}f = 0 $ must be invariant by translation.
	By the ergodicity of $ (T^x)_{x \in \Z} $ (assumption \textit{c}), such an $ f $ is constant.
	As a result $ (Z(t), t \geq 0) $ is a stationary and ergodic Markov process with reversible measure $ \pi(\omega) \mu(d\omega) $.
	
	As a consequence, for any $ t \geq 0 $,
	\begin{align*}
	\langle M^\omega_n \rangle_t \cvgas[a.s.]{n} t \int_\Omega h(\omega) \pi(\omega) \mu(d\omega), \quad \mu(d\omega)-a.s.
	\end{align*}
	This, together with the observation that
	\begin{align*}
	\sup_{t \geq 0} \abs{ M^\omega_n(t) - M^\omega_n(t^-) } \leq \frac{K^2}{\sqrt{n}}
	\end{align*}
	imply that $ (M^\omega_n(t))_{t\geq 0} $ converges in distribution in $ \sko[\R_+]{\R} $ to \Bm with variance $ \mean{h \pi} $ for $ \mu $ almost every $ \omega $ \citep[Proposition~1]{rebolledo_central_1980}.
	To conclude the proof of Theorem~\ref{thm:functional_clt}, we note that
	\begin{align*}
	\left( \mean{\frac{1}{N T^1 N}} \right)^{-2} \int_\Omega h(\omega) \pi(\omega) \mu(d\omega) = 2 m \left( \mean{\frac{1}{N T^1 N}} \mean{N \Nm} \right)^{-1} = \sigma^2.
	\end{align*}
\end{proof}

Let us now prove Lemma~\ref{lemma:remainder}.

\begin{proof}[Proof of Lemma~\ref{lemma:remainder}]
	This proof follows closely that of Proposition~4.2.2 in \cite{lam_les_2012}.
	From the convergence \eqref{scale_function}, we see that for any $ \varepsilon > 0 $, there exists $ M_\varepsilon = M_\varepsilon(\omega) $ such that, for all $ x \in \Z $,
	\begin{align*}
	\abs{x} > M_\varepsilon \implies \abs{ \frac{F(\omega, x)}{x} - \mean{\frac{1}{N T^1N}} } \leq \varepsilon.
	\end{align*}
	Set $ c = \mean{\frac{1}{N T^1N}} $.
	Hence, for any such $ x $,
	\begin{align} \label{bound_F_1}
	\abs{x} - \abs{\frac{F(\omega,x)}{c}} \leq \abs{\frac{F(\omega,x)}{c} - x} \leq \frac{\varepsilon \abs{x}}{c}.
	\end{align}
	Taking $ \varepsilon < c $, we have
	\begin{align*}
	\abs{x} \left( 1 - \frac{\varepsilon}{c} \right) \leq \abs{\frac{F(\omega,x)}{c}}
	\end{align*}
	and hence
	\begin{align} \label{bound_F_2}
	\frac{\varepsilon}{c} \abs{x} \leq \frac{\varepsilon \abs{F(\omega,x)}}{c(c-\varepsilon)}.
	\end{align}
	Combining \eqref{bound_F_1} and \eqref{bound_F_2}, we obtain
	\begin{align*}
	\forall \abs{x} > M_\varepsilon, \quad \abs{\frac{F(\omega,x)}{c} - x} \leq \frac{\varepsilon \abs{F(\omega,x)}}{c(c-\varepsilon)}.
	\end{align*}
	Let
	\begin{align*}
	H_\varepsilon = H_\varepsilon(\omega) = \sup_{\abs{x} \leq M_\varepsilon} \frac{\abs{F(\omega,x)}}{c}.
	\end{align*}
	Then for all $ x \in \Z $,
	\begin{align} \label{bound_F_3}
	\abs{\frac{F(\omega,x)}{c} - x} \leq \frac{\varepsilon \abs{F(\omega,x)}}{c(c-\varepsilon)} \wedge (M_\varepsilon + H_\varepsilon).
	\end{align}
	Replacing $ x $ with $ \xi_{nt} $, dividing by $ \sqrt{n} $ and squaring both sides, we obtain
	\begin{align} \label{bound_Rn}
	R^\omega_n(t)^2 \leq \frac{\varepsilon^2}{c^2(c-\varepsilon)^2} M^\omega_n(t)^2 + \frac{(M_\varepsilon + H_\varepsilon)^2}{n}.
	\end{align}
	By Doob's inequality,
	\begin{align*}
	\Eq[\xi^n_0]{\sup_{t \in [0,T]} M^\omega_n(t)^2 } &\leq 2 \Eq[\xi^n_0]{M^\omega_n(T)^2} \\
	&\leq 2 T \sup_{\omega\in \Omega} h(\omega) + 2 (M^\omega_n(0))^2 \\
	&\leq \frac{4}{3} m T K^4 + C,
	\end{align*}
	for some $C > 0$ using \eqref{qvar_Mn} and \eqref{uniform_ellipticity} (the bound on $M^\omega_n(0)$ comes from the fact that $ \abs{F(\omega,x)} \leq C \abs{x} $ and that $ \xi^n_0 $ converges).
	As a result, the expectation of the first term on the right hand side of \eqref{bound_Rn} can be made arbitrarily small, uniformly in $ t \in [0,T] $, by choosing $ \varepsilon $ small enough, while the second term vanishes as $ n \to \infty $ for each $ \varepsilon > 0 $.
	Hence
	\begin{align*}
	\lim_{n \to \infty} \Eq[\xi^n_0]{ \sup_{t \in [0,T]} R^\omega_n(t)^2 } = 0 \quad \mu(d\omega)-a.s.
	\end{align*}
\end{proof}

In passing, we can prove the following, which will be useful later in the proof of Theorem~\ref{thm:cvg-SHE-WF-noise}.

\begin{lemma} \label{lemma:bound_variance_RW}
	For $ \mu $-almost every $ \omega \in \Omega $,
	\begin{align*}
	\sup_{t \geq 0} \: \sup_{x \in \Z} \: \Eq[x]{\frac{\abs{\xi_{t} - x}^2}{t}} < \infty.
	\end{align*}
\end{lemma}

\begin{proof}
	By assumption \eqref{uniform_ellipticity}, for all $k, k' \in \Z$,
	\begin{align*}
	\abs{ F(\omega, k) - F(\omega, k') } \geq \frac{1}{K^2} \abs{k - k'}.
	\end{align*}
	Hence
	\begin{align*}
	\Eq[x]{ \abs{ \xi_t - x }^2 } \leq K^2 \Eq{ \abs{ M_1^\omega(t) - M^\omega_1(0) }^2 }.
	\end{align*}
	By \eqref{qvar_Mn} and using \eqref{uniform_ellipticity} again, there exists a constant $ C > 0 $ such that, for all $ t\geq 0$,
	\begin{align*}
	\Eq{ \abs{ M_1^\omega(t) - M_1^\omega(0) }^2 } \leq C t.
	\end{align*}
	The result then follows.
\end{proof}

\subsection{The environment viewed from the two random walks} \label{subsec:two_particles}

We now introduce an auxiliary process which records the environment viewed by the two independent random walks when they meet.
Let $ (\xi^1_t)_{t\geq 0} $ and $ (\xi^2_t)_{t\geq 0} $ be two independent random walks on $ \Z $ with transition rates given by \eqref{transition_rates} and started from $ x_1 $ and $ x_2 $ respectively.
Define, for $ t\geq 0 $, the time spent together by the two random walks up to time $ t $,
\begin{align} \label{definition_L0}
L_0(t) = \int_{0}^{t} \1{\xi^1_s = \xi^2_s} ds.
\end{align}
Also let $ L_0^{-1}(\cdot) $ be the right-continuous inverse of $ L_0 $, \textit{i.e.}
\begin{align*}
L_0^{-1}(t) = \inf \left\lbrace s \geq 0 : L_0(s) > t \right\rbrace.
\end{align*}
Note that
\begin{align} \label{time_change_justif}
\lbrace L_0^{-1}(t), t \geq 0 \rbrace = \lbrace t \geq 0 : \xi^1_t = \xi^2_t \rbrace.
\end{align}
As in \eqref{def_Z}, let $ (Z(t), t \geq 0) $ be the environment viewed from the first random walk, \textit{i.e.}
\begin{align*}
Z(t) = T^{\xi^1_t} \omega.
\end{align*}
In words, each time the first random walk jumps, we translate the whole environment in order to always keep the random walk at the origin.
We now want to keep track of the environment viewed from the two random walks.
Of course, when the two random walks are in different locations, this does not make sense, so whenever the two random walks are not together, we "fast-forward" to the next time when they are in the same location in $ \Z $, and we translate the whole environment so that they are both at the origin, until one of them jumps again, and we go to the next time when they rejoin, \textit{etc}.
This process is then a time change of $ (Z(t), t \geq 0) $, and in view of \eqref{time_change_justif}, the corresponding time change is exactly given by $ L_0^{-1}(\cdot) $.
This leads to the following definition of the environment viewed from the two random walks:
\begin{align*}
Y(t) = Z(L_0^{-1}(t)) \quad t\geq 0.
\end{align*}
The main result of this subsection is the following.
It gives a description of the typical environment seen by the two random walks whenever they are at the same location, and will be crucial to determine their overall coalescence rate.

\begin{proposition} \label{prop:ergodicity_environment_two_particles}
	The process $ (Y(t))_{t \geq 0} $, started from the initial distribution
	\begin{align} \label{stationary_measure_two_particles}
	\frac{\pi(\omega)^2}{\mean{\pi^2}} \mu(d\omega),
	\end{align}
	is a stationary and ergodic Markov process.
\end{proposition}

\begin{proof}
	Let $\tau_0$ denote the first time at which the two random walks $\xi^1$ and $\xi^2$ meet, \textit{i.e.}
	\begin{align*}
	\tau_0 = \inf \lbrace t \geq 0 : \xi^1_t = \xi^2_t \rbrace.
	\end{align*}
	Let $\mathcal{F}_0(\Z)$ be the space of real-valued functions on $\Z$ which are zero on all but a finite number of points.
	For $f \in \mathcal{F}_0(\Z)$, define
	\begin{align*}
	E^\omega f (x,y) = \Eq[x,y]{ f(\xi^1_{\tau_0}) }.
	\end{align*}
	For $x \in \Z$, $z \in \{ -1, 1 \}$ and $\omega \in \Omega $, let
	\begin{align} \label{definition_pz}
	j(\omega, x, z) = m \frac{N(\omega, x+z)}{\Nm(\omega, x)}
	\end{align}
	be the rate at which $\xi^i$ jumps from $x$ to $x+z$.
	The process $( \xi^1_{L_0^{-1}(t)}, t \geq 0 )$ then jumps from $x$ to $y$ at rate
	\begin{align*}
	q(\omega, x, y) = 2 \sum_{z \in \{-1,1\}} j(\omega, x, z) E^\omega \delta_y (x, x+z),
	\end{align*}
	where $\delta_y(x) = 1 $ if $ x = y $ and 0 otherwise.
	Setting $q(\omega, y) = q(\omega, 0, y)$, the generator of $(Y(t))_{t\geq 0}$ takes the form
	\begin{align*}
	\mathcal{L}^Y f(\omega) = \sum_{y \in \Z} q(\omega, y) \left( f(T^y \omega) - f(\omega) \right).
	\end{align*}
	Proposition~\ref{prop:ergodicity_environment_two_particles} will then be proved if we show that for $\mu$ almost every $\omega \in \Omega$ and every $y \in \Z$,
	\begin{align} \label{reversibility_two_particles}
	\pi(\omega)^2 q(\omega, y) = \pi(T^y\omega)^2 q(T^y\omega, -y).
	\end{align}
	Indeed the stationarity of the measure~\eqref{stationary_measure_two_particles} follows directly from \eqref{reversibility_two_particles} and ergodicity is a consequence of assumption \textit{c} and
	\begin{align*}
	\int_\Omega f(\omega) (-\mathcal{L}^Y) f(\omega) \pi(\omega)^2 \mu(d\omega) = \frac{1}{2} \sum_{y \in \Z} \int_\Omega q(\omega, y) \left( f(T^y\omega) - f(\omega) \right)^2 \pi(\omega)^2 \mu(d\omega).
	\end{align*}
	
	We now prove \eqref{reversibility_two_particles}.
	To avoid writing infinite sums, we first restrict ourselves to random walks on a bounded region $\lbrace -A, \ldots, A \rbrace$.
	For $A \geq 1$, define
	\begin{equation*}
	j^A(\omega,x,z) = \left\lbrace
	\begin{aligned}
	j(\omega,x,z) \hphantom{0} & \text{ if } \abs{x} \leq A \text{ and } \abs{x+z} \leq A, \\
	0 \hphantom{p(x,z)} & \text{ otherwise.}
	\end{aligned}
	\right.
	\end{equation*}
	Let $(\xi^{A,1}_t)_{t \geq 0}$ and $(\xi^{A,2}_t)_{t\geq 0}$ be two independent random walks on $\lbrace -A, \ldots, A \rbrace$ which jump from $x$ to $x+z$ at rate $j^A(\omega,x,z)$ (\textit{i.e.} they behave as $\xi^i$ but the jumps leading outside $\lbrace -A, \ldots, A \rbrace$ are suppressed).
	Let $\tau_0^A$ be the first time at which $\xi^{A,1}$ and $\xi^{A,2}$ meet and define, for $f \in \mathcal{F}_0(\Z)$, $x, y \in \lbrace -A, \ldots, A \rbrace$,
	\begin{align*}
	E^{A,\omega} f(x,y) = \Eq[x,y]{ f( \xi^{A,i}_{\tau_0^A} ) }.
	\end{align*}
	Then, since
	\begin{align*}
	\pi(\omega,x) j^A(\omega,x,z) = \pi(\omega,x+z) j^A(\omega,x+z,-z),
	\end{align*}
	for any $f, g \in \mathcal{F}_0(\Z)$,
	\begin{multline*}
	\sum_{x= -A}^{A} \sum_{z \in \{-1,1\}} \pi(\omega,x) j^A(\omega,x,z) E^{A,\omega} f(x+z,y) E^{A,\omega} g(x,y) \\ = \sum_{x=-A}^{A} \sum_{z \in \{ -1,1 \}} \pi(\omega,x) j^A(\omega,x,z) E^{A,\omega} f(x,y) E^{A,\omega} g(x+z,y).
	\end{multline*}
	As a result,
	\begin{multline} \label{ugly_sum}
	\sum_{x, y \in [-A,A]} \sum_{z \in \{-1,1\} } \pi(\omega,x)\pi(\omega,y) \Big[ j^A(\omega,x,z) \left( E^{A,\omega} f(x+z,y) - E^{A,\omega} f(x, y) \right) \\ + j^A(\omega,y,z) \left( E^{A,\omega} f(x, y+z) - E^{A,\omega} f(x,y) \right) \Big] E^{A,\omega} g(x,y) \\[1.3em] = \sum_{x, y \in [-A,A]} \sum_{z \in \{-1,1\} } \pi(\omega,x) \pi(\omega,y) E^{A,\omega} f(x,y) \Big[ j^A(\omega,x,z) \left( E^{A,\omega} g(x+z,y) - E^{A,\omega} g(x,y) \right) \\ + j^A(\omega,y,z) \left( E^{A,\omega} g(x, y+z) - E^{A,\omega} g(x,y) \right) \Big].
	\end{multline}
	But, by the Markov property and the definition of $ E^\omega $, for any $ x \neq y \in \lbrace -A, \ldots, A \rbrace $ and for any $ f \in \mathcal{F}_0(\Z) $,
	\begin{multline*}
	\sum_{z \in \{-1,1\}} \Big[ j^A(\omega,x,z) \left( E^{A,\omega} f(x+z,y) - E^{A,\omega} f(x,y) \right) \\ + j^A(\omega,y,z) \left( E^{A,\omega} f(x,y+z) - E^{A,\omega} f(x,y) \right) \Big] = 0.
	\end{multline*}
	As a result, the only non-zero terms in the sums appearing in \eqref{ugly_sum} are those for which $ x = y $.
	We thus obtain that, for any $ f, g \in \mathcal{F}_0(\Z)$, 
	\begin{multline*}
	\sum_{x = -A}^{A} \sum_{z \in \{-1, 1\} } \pi(\omega,x)^2 j^A(\omega,x,z) g(x) \left( E^{A,\omega} f(x+z, x) - f(x) \right) \\ = \sum_{x = -A}^{A} \sum_{z \in \{-1, 1\}} \pi(\omega,x)^2 j^A(\omega,x,z) f(x) \left( E^{A,\omega} g(x+z,x) - g(x) \right).
	\end{multline*}
	Taking $ f = \delta_y $ and $ g = \delta_x $, we obtain
	\begin{align*}
	\pi(\omega,x)^2 \sum_{z \in \{-1, 1\}} j^A(\omega,x,z) E^{A,\omega} \delta_y(x+z,x) = \pi(\omega,y)^2 \sum_{z \in \{-1, 1\}} j^A(\omega,y,z) E^{A,\omega} \delta_x(y+z,y).
	\end{align*}
	We now wish to let $A \to \infty$ in order to obtain \eqref{reversibility_two_particles}.
	Clearly, for $A$ large enough, $j^A(\omega,x,z) = j(\omega,x,z)$.
	Let $T_A$ be such that
	\begin{align*}
	T_A = \inf \lbrace t \geq 0 : \abs{\xi^{i}_t} > A \text{ for some } i \in \{ 1, 2 \} \rbrace.
	\end{align*}
	Then, since $ T_A \to \infty $ almost surely as $ A \to \infty $, $ \xi^{A,i}_{\tau_0^A} $ converges in distribution to $ \xi^i_{\tau_0} $.
	Hence $E^{A,\omega} \delta_y(x, x+z) \to E^\omega \delta_y(x, x+z)$ and \eqref{reversibility_two_particles} is proved.
\end{proof}

\subsection{Convergence to the Brownian flow with delayed coalescence} \label{subsec:proof_cvg_brownian_flow}

We now set out to prove Theorem~\ref{thm:cvg_brownian_flow}.
Recall the construction of $ (\A^{\omega,\lambda})_{t \geq 0} $ started from two lineages in Subsection~\ref{subsec:dual}.
In particular, recall that $ (\xi^1_t)_{t \geq 0} $ and $ (\xi^2_t)_{t \geq 0} $ are two independent random walks on $ \Z $ with transition rates given by \eqref{transition_rates} and
\begin{align*}
L(t) = \int_{0}^{t} \frac{\1{\xi^1_s = \xi^2_s}}{\lambda N(\omega, \xi^1_s)} ds,
\end{align*}
and that the two lineages coalesce at time $ T_c $ such that
\begin{align*}
T_c = \inf \{ t \geq 0 : L(t) > E \}
\end{align*}
where $ E $ is an independent exponential random variable with parameter 1.

For $ n\geq 1 $, we have set
\begin{align*}
\bm{\A}^n_t = \frac{1}{\sqrt{n}} \A^{\omega, \sqrt{n}}_{nt}.
\end{align*}
By Theorem~\ref{thm:functional_clt}, $ (\frac{1}{\sqrt{n}} \xi^i_{nt})_{t \geq 0} $ converges $ \mu $-almost surely in distribution in $ \sko{\R} $ to \Bm with variance $ \sigma^2 $ given by \eqref{parameters}.
Since $ \xi^1 $ and $ \xi^2 $ are independent, by the Skorokhod representation theorem, for $ \mu $-almost every $ \omega $, there exists a probability space on which two sequences of processes $ (\xi^{1,n}_t, \xi^{2,n}_t)_{t \in [0,T]} $ and two \Bm* $ (X^1_t, X^2_t)_{t \in [0,T]} $ are defined and such that
\begin{enumerate}[i)]
	\item for all $ n \geq 1 $, $ (\xi^{i,n}_t)_{t \in [0,T]} $ is distributed like $ (\xi^i_t)_{t \in [0,T]} $, for $ i \in \{1, 2\} $,
	\item $ \xi^{1,n} $ and $ \xi^{2,n} $ are independent for all $ n \geq 1 $,
	\item the sequence of processes $ (\frac{1}{\sqrt{n}} \xi^{i,n}_{nt})_{t\in [0,T]} $ converges to $ (X^i_t)_{t \in [0,T]} $ in $ \sko{\R} $ almost surely, for $ i \in \{1,2\} $.
\end{enumerate}

For $ n \geq 1 $, define
\begin{align*}
L_n(t) = \frac{1}{\sqrt{n}} \int_{0}^{nt} \frac{\1{\xi^{1,n}_s = \xi^{2,n}_s}}{N(\omega, \xi^{1,n}_s)} ds.
\end{align*}
Then $ (L_n(t))_{t\geq 0} $ is distributed like $ (L(nt))_{t \geq 0} $ (with $ \lambda = \sqrt{n} $).

\begin{lemma} \label{lemma:convergence_local_time}
	As $ n \to \infty $, for $ \mu $-almost every $ \omega \in \Omega $, $ (L_n(t), t \in [0,T]) $ converges in probability to $ ( \gamma L^0_t(X^1-X^2), t \in [0,T]) $ in the Skorokhod topology, where $ L^0_t(X^1-X^2) $ is the local time at zero of $ X^1-X^2 $ and $ \gamma $ is given by \eqref{parameters}.
\end{lemma}

\begin{proof}
	For $ t \geq 0 $ and $ n \geq 1 $, let
	\begin{align*}
	L^0_n(t) = \int_{0}^{t} \1{\xi^{1,n}_s = \xi^{2,n}_s} ds.
	\end{align*}
	Then $ (L^0_n(t), t \geq 0) $ is distributed like $ (L_0(t), t\geq 0) $ of \eqref{definition_L0}; hence
	\begin{align*}
	Y^n(t) = T^{\xi^{1,n}_{(L^{0}_n)^{-1}(t)}}\omega, t \geq 0,
	\end{align*}
	is distributed like $ (Y(t), t\geq 0) $.
	We can then write $ L_n(t) $ as
	\begin{align*}
	L_n(t) = \frac{1}{\sqrt{n}} L^0_n(nt) \frac{1}{L^0_n(nt)} \int_{0}^{L^0_n(nt)} \frac{1}{N(Y^n(s))} ds.
	\end{align*}
	Since $ (\frac{1}{\sqrt{n}} \xi^{i,n}_{nt}, t \in [0,T]) $ converges to $ (X^i_t)_{t \in [0,T]} $ almost surely in $ \sko{\R} $, $ (\frac{1}{\sqrt{n}} L^0_n(nt)$, $ t \in [0,T]) $ converges almost surely to $ (L^0_t(X^1-X^2), t \in [0,T]) $ in the uniform topology.
	Furthermore, by Proposition~\ref{prop:ergodicity_environment_two_particles} and by the pointwise ergodic theorem, for each $ n \geq 1 $,
	\begin{align*}
	\frac{1}{t} \int_{0}^{t} \frac{1}{N(Y^n(s))} ds \: \cvgas{t} \: \int_\Omega \frac{1}{N(\omega)} \frac{\pi(\omega)^2}{\mean{\pi^2}} \mu(d\omega) = \gamma \quad \mu(d\omega)-a.s.
	\end{align*}
	Fix $ \varepsilon > 0 $ and $ T > 0 $.
	For each $ n \geq 1 $, there exists a (random) $ t_0^n > 0 $ such that
	\begin{align*}
	\forall t \geq t_0^n, \quad \abs{\frac{1}{t} \int_{0}^{t} \frac{1}{N(Y^n(s))} ds - \gamma } \leq \varepsilon.
	\end{align*}
	Furthermore, there exists $ n_0 \geq 1 $ such that, for any $ n \geq n_0 $,
	\begin{align*}
	\sup_{t \in [0,T]} \abs{\frac{1}{\sqrt{n}} L^0_n(nt) - L^0_t(X^1-X^2)} \leq \varepsilon.
	\end{align*}
	As a result, for $ n \geq n_0 $,
	\begin{align*}
	\abs{L_n(t) - \gamma L^0_t(X^1-X^2)} &\leq \frac{1}{\sqrt{n}} L^0_n(nt) \left( \varepsilon + 2K \1{L^0_n(nt) < t_0^n} \right) + \gamma \varepsilon \\
	& \leq \varepsilon (L^0_t(X^1-X^2) + \varepsilon) + 2K \frac{1}{\sqrt{n}} t_0^n + \gamma \varepsilon.
	\end{align*}
	Since the $ (t^n_0)_{n \geq 1} $ are identically distributed, $ \frac{1}{\sqrt{n}} t^n_0 $ converges to 0 in probability as $ n \to \infty $, and the result follows.
\end{proof}

Let us now prove Theorem~\ref{thm:cvg_brownian_flow}.

\begin{proof}[Proof of Theorem~\ref{thm:cvg_brownian_flow}]
	Define, on the same probability space as the $ \xi^{i,n} $ an independent exponential random variable $ E $ with parameter 1, and set
	\begin{align*}
	T^n_c = \inf \{ t \geq 0 : L_n(t) > E \}.
	\end{align*}
	Then
	\begin{equation*}
	\A^n_t = \left\lbrace
	\begin{aligned}
	\left\lbrace \frac{1}{\sqrt{n}} \xi^{1,n}_{nt}, \frac{1}{\sqrt{n}} \xi^{2,n}_{nt} \right\rbrace & \text{ if } t < T^n_c, \\
	\left\lbrace \frac{1}{\sqrt{n}} \xi^{1,n}_{nt} \right\rbrace \hspace{40pt} & \text{ if } t \geq T^n_c,
	\end{aligned}
	\right.
	\end{equation*}
	is distributed as the process $ \bm{\A}^n $ of Theorem~\ref{thm:cvg_brownian_flow}.
	Lemma~\ref{lemma:convergence_local_time} implies
	\begin{align*}
	T^n_c \cvgas[\mathbb{P}^\omega]{n} T_c,
	\end{align*}
	where
	\begin{align*}
	T_c = \inf \{ t \geq 0 : \gamma L^0_t(X^1-X^2) > E \}.
	\end{align*}
	As a result, $ (\A^n_t, t \in [0,T]) $ converges in probability in $ \sko{\R \cup \R^2} $ to the Brownian flow with delayed coalescence with paramters $ (\sigma^2, \gamma) $.
	Hence, $ (\bm{\mathcal{A}}_t^n, t \in [0,T]) $ converges in distribution to the Brownian flow with delayed coalescence with these parameters.
\end{proof}

\section{Convergence to the stochastic heat equation with Wright-Fisher noise} \label{sec:proof_cvg_SHE}

\subsection{Proof of the main result} \label{subsec:proof_main_result}

Let us now prove the convergence of the rescaled stepping stone model in a random environment to the stochastic heat equation with Wright-Fisher noise.
The proof follows a classical approach: we prove that the sequence $ (\bm{p}^n_t(\omega,\cdot), t \in [0,T]) $ is tight in the space of continuous $ \Xi $-valued functions $ \mathcal{C}([0,T], \Xi) $ and we use the duality relation (Proposition~\ref{prop:duality}) together with Theorem~\ref{thm:cvg_brownian_flow} to identify the limit.

\begin{lemma}[Tightness of the sequence] \label{lemma:tightness}
	For all $ T > 0 $, the sequence $ (\bm{p}^n_t(\omega, \cdot), t \in [0,T]) $ defined in Subsection~\ref{subsec:main_result} is tight in $ \mathcal{C}([0,T], \Xi) $, $ \mu(d\omega) $-almost surely.
\end{lemma}

We prove this lemma in the next subsection.
For now, let us conclude the proof of Theorem~\ref{thm:cvg-SHE-WF-noise}.

\begin{proof}[Proof of Theorem~\ref{thm:cvg-SHE-WF-noise}]
	Let $ (\bm{p}^{n_j})_{j \geq 1} $ be a subsequence of $ (\bm{p}^n)_{n\geq 1} $ converging in distribution to $ \bm{p}^\infty \in \mathcal{C}([0,T], \Xi) $.
	In view of the definition of $ \bm{p}^n $ and $ \bm{\A}^n $, by Proposition~\ref{prop:duality} and \eqref{integral_pn}, for all $ n \geq 1 $, $ k \geq 1 $ and $ \phi_1, \ldots, \phi_k \in \mathcal{C}^\infty_c(\R) $,
	\begin{align} \label{duality_pn}
	\Eq[\bm{p}^n_0]{\prod_{i=1}^{k} \dual{\bm{p}^n_t}{\phi_i} } = \frac{1}{n^{k/2}} \sum_{(x_1, \ldots, x_k) \in \Z^k} \Eq[\left\lbrace \frac{x_1}{\sqrt{n}}, \ldots, \frac{x_k}{\sqrt{n}} \right\rbrace]{ \prod_{i=1}^{N_t} \bm{p}_0^n(\xi^{i,n}_t) } \phi_1\left( \frac{x_1}{\sqrt{n}} \right) \ldots \phi_k\left( \frac{x_k}{\sqrt{n}} \right),
	\end{align}
	where $ \bm{\A}^n_t = \{ \xi^{1,n}_t, \ldots, \xi^{N_t,n}_t \} $.
	Furthermore, for all $ n \geq 1 $, $ t \geq 0 $, 
	\begin{align} \label{bound_pn}
	\abs{ \langle \bm{p}^n_t, \phi \rangle } \leq \frac{1}{\sqrt{n}} \sum_{x \in \Z} \abs{\phi \left( \frac{x}{\sqrt{n}} \right)}.
	\end{align}
	Since the right hand side is bounded uniformly in $ n $ for any $ \phi \in \mathcal{C}^\infty_c(\R) $, the set
	\begin{align} \label{compact_containment}
	\left\lbrace p \in \Xi : \abs{\dual{p}{\phi}} \leq \sup_{k \geq 1} \frac{1}{\sqrt{k}} \sum_{x \in \Z} \abs{\phi \left( \frac{x}{\sqrt{k}} \right)}, \forall \phi \in \mathcal{C}^\infty_c(\R) \right\rbrace.
	\end{align}
	is strongly bounded, hence compact in $ \Xi $ for the vague topology by the Banach-Alaoglu theorem.
	As a consequence $ \bm{p}^\infty $ also takes values in this subset.
	Since $ p \mapsto \dual{p} $ is continuous and bounded on this subset for any smooth and bounded $ \phi : \R \to \R $, we can let $ n = n_j \to \infty $ in the left hand side of \eqref{duality_pn}.
	The right hand side can also be written
	\begin{align*}
	\int_{\R^k} \Eq[\left\lbrace \frac{\lfloor \sqrt{n} x_1 \rfloor}{\sqrt{n}}, \ldots, \frac{\lfloor \sqrt{n} x_k \rfloor}{\sqrt{n}} \right\rbrace]{ \prod_{i = 1}^{N_t} \bm{p}_0^n(\xi^{i,n}_t) } \phi_1 \left( \frac{\lfloor \sqrt{n} x_1 \rfloor}{\sqrt{n}} \right) \ldots \phi_k \left( \frac{\lfloor \sqrt{n} x_k \rfloor}{\sqrt{n}} \right) dx_1 \ldots dx_k.
	\end{align*}
	Since $ \prod_{i = 1}^{N_t} \bm{p}_0^n(\xi^{i,n}_t) \leq 1 $, we can use Theorem~\ref{thm:cvg_brownian_flow} and the dominated convergence theorem for $ k \in \{1, 2\} $ to obtain
	\begin{align} \label{duality_pinfty}
	\Eq[\bm{p}_0]{\prod_{i=1}^{k} \dual{\bm{p}^\infty_t}{\phi_i} } = \int_{\R^k} \Eq[\{ x_1, x_k \}]{ \prod_{i=1}^{N_t} \bm{p}_0(X^{i}_t) } \phi_1(x_1) \phi_k(x_k) dx_1 dx_k,
	\end{align}
	where $ \mathcal{D}_t = \{ X^1_t, X^{N_t}_t \} $ is the Brownian flow with delayed coalescence with parameters $ \sigma^2 $ and $ \gamma $, started from either $ \{ x_1 \} $ if $ k = 1 $ or $ \{ x_1, x_2 \} $ if $ k = 2 $.
	Note that, passing to the limit in \eqref{bound_pn} and since $ p^n_t \geq 0 $, we have $ 0 \leq \langle \bm{p}^\infty_t, \phi \rangle \leq \int_\R \phi(x) dx $ for all $ \phi \geq 0 $.
	We have thus proved that $ (\bm{p}^\infty_t, t \geq 0) $ satisfies the conditions of Proposition~\ref{prop:duality_brownian_flow}, hence it is a solution to the martingale problem \eqref{martingale_pb_SHE} and the proof of Theorem~\ref{thm:cvg-SHE-WF-noise} is complete.
\end{proof}

\subsection{Tightness of the sequence} \label{subsec:tightness}

Since the sequence of processes $ (\bm{p}^n_t(\omega,\cdot), t \in [0,T]) $ takes values in a compact subspace of $ \Xi $ almost surely (see \eqref{compact_containment} above), to prove that it is tight in $ \mathcal{C}([0,T], \Xi) $ it is enough to show that for any $ \phi \in \mathcal{C}^\infty_c(\R) $ and $ f \in C^2(\R) $, the sequence of real-valued processes $ ( f(\dual{\bm{p}^n_t}), t\in [0,T]) $ is tight in $ \mathcal{C}([0,T],\R) $ \citep[Theorem~3.9.1]{ethier_markov_1986}.
To do so, we want to use the Aldous-Rebolledo criterion and show that both the finite variation part and the quadratic variation part of $ ( f(\dual{\bm{p}^n_t}), t\in [0,T]) $ are tight.
In fact, we show that this is true for another process $ ( f(\dual{\tilde{p}^n_t}), t\in [0,T]) $, defined below, and we show that the difference between the two processes vanishes as $ n \to \infty $.
Let us first state two lemmas.

\begin{lemma} \label{lemma:holder_pn}
	There exists $ \beta \in (0,1) $ such that for all $ T > 0$ there exist random variables $ (C_n, n\geq 1) $ such that, for all $n \geq 1$,
	\begin{align*}
	\sup_{t \in [0,T]} \sup_{x, y \in \Z} n^{\beta/2} \frac{\abs{p^n_{nt}(\omega,x) - p^n_{nt}(\omega,y)}}{\abs{x - y}^\beta} \leq C_n \quad \text{ almost surely,}
	\end{align*}
	and there exists $ C > 0 $ such that, for all $n \geq 1$, $ \Eq{C_n^2} \leq C $, $\mu(d\omega)$-almost surely.
\end{lemma}

Lemma~\ref{lemma:holder_pn} is proved in Subsection~\ref{subsec:holder_estimate}.
Recall the definition of $ \pi $ in \eqref{reversible_measure}.
For $ n \geq 1 $, $\delta > 0$ and $ x \in \frac{1}{\sqrt{n}} \Z $, set
\begin{align} \label{definition_Pi}
\Pi^{n,\delta}(\omega,x) = \frac{1}{2 \delta \sqrt{n}} \sum_{k \in B(\sqrt{n} x, \delta \sqrt{n}) \cap \Z} \pi(\omega,k)
\end{align}
and extend $ x \mapsto \Pi^{n,\delta}(\omega,x) $ to $ \R $ by linear interpolation.

\begin{lemma} \label{lemma:uniform_cvg_pi}
	For any $\delta > 0$, $ \Pi^{n,\delta} $ converges to 1 uniformly on compact sets as $ n \to \infty $, $ \mu(d\omega) $-almost surely.
\end{lemma}

\begin{proof}
	By the pointwise ergodic theorem, for each $ x \in \R $ and $ \delta > 0 $,
	\begin{align*}
	\Pi^{n,\delta}(\omega,x) \cvgas{n} 1 \quad \mu(d\omega)-\text{almost surely.}
	\end{align*}
	Moreover, for any $ x, y \in \R $, by \eqref{uniform_ellipticity},
	\begin{align*}
	\abs{ \Pi^{n,\delta}(\omega,x) - \Pi^{n,\delta}(\omega,y) } \leq \frac{K^4}{\delta} \abs{x-y}.
	\end{align*}
	Hence $ x \mapsto \Pi^{n,\delta}(\omega,x) $ is equicontinuous $\mu(d\omega)$-almost surely.
	By Arzela-Ascoli's theorem, it converges uniformly on compact sets.
\end{proof}

We can now prove the tightness of the sequence.

\begin{proof}[Proof of Lemma~\ref{lemma:tightness}]
	For $ n \geq 1 $, define a $\Xi$-valued process $ ( \tilde{p}^n_t, t \geq 0 ) $ by
	\begin{align*}
	\dual{\tilde{p}^n_t} = \frac{1}{\sqrt{n}} \sum_{x \in \Z} \, \pi(\omega, x) \, p^n_{nt}(\omega, x) \, \phi \left( \frac{F(\omega, x)}{c \sqrt{n}} \right),
	\end{align*}
	where $ F(\omega, x) $ was defined in \eqref{definition_F} and $ c $ was defined just above \eqref{bound_F_1}.
	We first show that for any $ \phi \in \mathcal{C}^\infty_c (\R) $ and $ f \in C^2(\R) $ the sequence $ ( f(\dual{\tilde{p}^n_t}), t \in [0,T] ) $ is tight and then we show that the difference between $ \dual{\bm{p}^n_t} $ and $\dual{\tilde{p}^n_t} $ vanishes as $n \to \infty$.
	
	\paragraph*{Tightness of $ ( f(\dual{\tilde{p}^n_t}), t\in [0,T]) $}
	
	By the definition of $p^n$ and \eqref{stepping_stone_equation},
	\begin{multline*}
	d \dual{\tilde{p}^n_t} = n \frac{m}{\sqrt{n}} \sum_{x \in \Z} \sum_{z \in \{-1,1\} } \pi(\omega,x) j(\omega,x,z) \left( p^n_{nt}(x+z) - p^n_{nt}(x) \right) \phi \left( \frac{F(\omega, x)}{c \sqrt{n}} \right) dt \\ + \frac{1}{\sqrt{n}} \sum_{x \in \Z} \pi(\omega,x) \phi \left( \frac{F(\omega, x)}{c \sqrt{n}} \right) \sqrt{ \frac{\sqrt{n}}{N(\omega,x)} p^n_{nt}(x) (1 - p^n_{nt}(x) ) } d B^x_t,
	\end{multline*}
	where $j(\omega,x,z)$ was defined in \eqref{definition_pz} and $ ( B^x, x \in \Z ) $ is a family if independent Brownian motions.
	Note that, by the reversibility of $ \pi $,
	\begin{multline} \label{tightness_1}
	n \frac{m}{\sqrt{n}} \sum_{x \in \Z} \sum_{z \in \{-1,1\} } \pi(\omega,x) j(\omega,x,z) \left( p^n_{nt}(x+z) - p^n_{nt} (x) \right) \phi \left( \frac{F(\omega, x)}{c \sqrt{n}} \right) \\ = n \frac{m}{\sqrt{n}} \sum_{x \in \Z} \sum_{z \in \{-1,1\} } \pi(\omega,x) j(\omega,x,z) p^n_{nt}(x) \left( \phi \left( \frac{F(\omega, x + z)}{c\sqrt{n}} \right) - \phi \left( \frac{F(\omega, x)}{c \sqrt{n}} \right) \right).
	\end{multline}
	By Taylor's theorem,
	\begin{multline*}
	\phi \left( \frac{F(\omega, x + z)}{c\sqrt{n}} \right) - \phi \left( \frac{F(\omega, x)}{c \sqrt{n}} \right) = \phi' \left( \frac{F(\omega, x)}{c \sqrt{n}} \right) \frac{F(\omega, x+z) - F(\omega, x)}{c \sqrt{n}} \\ + R\left( \frac{F(\omega, x)}{c \sqrt{n}}, \frac{F(\omega, x+z)}{c \sqrt{n}} \right) \frac{( F(\omega, x+z) - F(\omega, x) )^2}{c^2 n},
	\end{multline*}
	where
	\begin{align*}
	R(x,y) = \int_0^1 (1-t) \phi''(x + t(y-x) ) dt.
	\end{align*}
	In view of \eqref{martingale_equation_F}, equation \eqref{tightness_1} equals
	\begin{align*}
	\frac{m}{c^2 \sqrt{n}} \sum_{\substack{x \in \Z \\ z \in \{-1,1\}}} \pi(\omega,x) j(\omega,x,z) p^n_{nt}(x) R\left( \frac{F(\omega, x)}{c \sqrt{n}}, \frac{F(\omega, x+z)}{c \sqrt{n}} \right) ( F(\omega, x+z) - F(\omega, x) )^2.
	\end{align*}
	Since $ \abs{p^n_{nt}(x)} \leq 1 $ and by assumption \eqref{uniform_ellipticity}, this is smaller than
	\begin{align*}
	\frac{C}{\sqrt{n}} \sum_{x \in \Z} \sum_{z \in \{-1,1\} } \abs{R\left( \frac{F(\omega, x)}{c \sqrt{n}}, \frac{F(\omega, x+z)}{c \sqrt{n}} \right)}
	\end{align*}
	for some constant $ C > 0 $.
	In turn, the above expression is uniformly bounded as $ n \to \infty $ since $\phi \in \mathcal{C}^\infty_c (\R)$ (also note \eqref{bound_F_3}).
	We have thus proved the existence of a constant $ C_1 > 0 $ such that, for all $n \geq 1$,
	\begin{align} \label{bound_finite_variation}
	\abs{n \frac{m}{\sqrt{n}} \sum_{x \in \Z} \sum_{z \in \{-1,1\} } \pi(\omega,x) j(\omega,x,z) \left( p^n_{nt}\left( x+z \right) - p^n_{nt} \left(x \right) \right) \phi \left( \frac{F(\omega, x)}{c \sqrt{n}} \right)} \leq C_1.
	\end{align}
	We now turn to the quadratic variation part of $ \dual{\tilde{p}^n_t} $.
	Since the Brownian motions $ (B^x, x \in \Z ) $ are independent, it is
	\begin{align*}
	\int_0^t \frac{1}{\sqrt{n}} \sum_{x \in \Z} \pi(\omega,x)^2 \phi \left( \frac{F(\omega, x)}{c \sqrt{n}} \right)^2 \frac{1}{N(\omega,x)} p^n_{ns}(x) (1 - p^n_{ns}(x) ) ds.
	\end{align*}
	Then, by \eqref{uniform_ellipticity} and since $ \phi \in \mathcal{C}^\infty_c(\R) $ there exists a constant $ C_2 > 0 $ such that
	\begin{align} \label{bound_quadratic_variation}
	\abs{\frac{1}{\sqrt{n}} \sum_{x \in \Z} \pi(\omega,x)^2 \phi \left( \frac{F(\omega, x)}{c \sqrt{n}} \right)^2 \frac{1}{N(\omega,x)} p^n_{ns}(x) (1 - p^n_{ns}(x) )} \leq C_2.
	\end{align}
	
	We then write $ FV_t(f, \phi) $ and $ QV_t(f, \phi) $ for the finite variation part and the quadratic variation part of $ f(\dual{\tilde{p}^n_t}) $.
	Then, by \eqref{bound_finite_variation} and \eqref{bound_quadratic_variation} and by the It\^o formula,
	\begin{align*}
	& \abs{ FV_t(f, \phi) - FV_s(f, \phi) } \leq \left( \sup \abs{f'} C_1 + \frac{1}{2} \sup \abs{f''} C_2 \right) \abs{t-s} \\
	& \abs{ QV_t(f, \phi) - QV_t(f, \phi) } \leq \sup \abs{f'}^2 C_2 \abs{t-s},
	\end{align*}
	where the supremum is taken over a suitable compact set (depending on $ \phi $).
	Hence both $ (FV_t(f, \phi), t \in [0,T] ) $ and $ ( QV_t(f, \phi), t \in [0,T] ) $ are tight, and by the Aldous-Rebolledo criterion, $ ( f(\dual{\tilde{p}^n_t}), t \in [0,T] ) $ is tight in $ \mathrm{D}([0,T], \Xi) $.
	Since for all $ n \geq 1 $, $ t \mapsto f(\langle \tilde{p}^n_t, \phi \rangle) $ is continuous almost surely, so are its potential limit points, hence the sequence is tight in $ \mathcal{C}([0,T], \Xi) $.
	
	\paragraph*{Conclusion}
	
	Let us now show that the difference between $ \dual{\tilde{p}^n_t} $ and $ \dual{p^n_t} $ vanishes as $n \to \infty$.
	First note that, by \eqref{bound_F_3}, for any $ \varepsilon > 0 $,
	\begin{multline*}
	\abs{ \frac{1}{\sqrt{n}} \sum_{x \in \Z} \pi(\omega,x) p^n_{nt}(x) \phi \left( \frac{F(\omega, x)}{c \sqrt{n}} \right) - \frac{1}{\sqrt{n}} \sum_{x \in \Z} \pi(\omega,x) p^n_{nt}(x) \phi\left( \frac{x}{\sqrt{n}} \right) } \\ \leq \frac{4}{\sqrt{n}} \sup_{x \in \R} \abs{\phi(x)} M_\varepsilon K^4 + \frac{1}{\sqrt{n}} \sum_{\substack{x \in \Z \\ \abs{x} > M_\varepsilon}} K^4 \sup_{\abs{y - x/\sqrt{n}} \leq \varepsilon \frac{\abs{x}}{c \sqrt{n}}} \abs{\phi'(y)} \varepsilon \frac{\abs{x}}{c \sqrt{n}}.
	\end{multline*}
	The first term on the right hand side vanishes as $n \to \infty$ for any fixed $\varepsilon$ while the second term can be made arbitrarily small for all $n \geq 1$ by choosing $\varepsilon$ small enough.
	Hence
	\begin{align*}
	\sup_{t \geq 0} \abs{ \dual{\tilde{p}^n_t} - \frac{1}{\sqrt{n}} \sum_{x \in \Z} \pi(\omega,x) p^n_{nt}(x) \phi\left( \frac{x}{\sqrt{n}} \right) } \cvgas{n} 0.
	\end{align*}
	To conclude, we write, for $\delta > 0$,
	\begin{multline*}
	\abs{ \frac{1}{\sqrt{n}} \sum_{x\in \Z} p^n_{nt}(x) \phi(x/\sqrt{n}) - \frac{1}{\sqrt{n}} \sum_{x \in \Z} \pi(\omega,x) p^n_{nt}(x) \phi(x/\sqrt{n}) } \\ \leq \frac{1}{\sqrt{n}} \sum_{x \in \Z} \abs{\phi(x/\sqrt{n})} \abs{ 1 - \frac{1}{2\delta \sqrt{n}} \sum_{y \in B(x, \delta \sqrt{n}) \cap \Z} \pi(\omega,y) } \\ + \frac{1}{\sqrt{n}} \sum_{x \in \Z} \frac{1}{2\delta \sqrt{n}} \sum_{y \in B(x, \delta \sqrt{n}) \cap \Z} \pi(\omega,y) \abs{ p^n_{nt}(x) \phi(x/\sqrt{n}) - p^n_{nt}(y) \phi(y/\sqrt{n}) }.
	\end{multline*}
	The first term on the right hand side vanishes as $n \to \infty$ for any $\delta > 0$ by Lemma~\ref{lemma:uniform_cvg_pi} (recall that $\phi$ is compactly supported).
	For the second term, by Lemma~\ref{lemma:holder_pn}, for $t \in [0,T]$ and $ \abs{x - y} \leq \delta \sqrt{n} $,
	\begin{align*}
	\abs{ p^n_{nt}(x) \phi(x/\sqrt{n}) - p^n_{nt}(y) \phi(y/\sqrt{n}) } \leq \sup_{\abs{z-x/\sqrt{n}} \leq \delta} \abs{\phi'(z)} \delta + \abs{\phi(x/\sqrt{n})} C_n \delta^\beta.
	\end{align*}
	Hence the second term can be made arbitrarily small for all $ n \geq 1 $ by choosing $\delta$ small enough.
	As a result,
	\begin{align*}
	\sup_{t \in [0,T]} \abs{ \dual{\tilde{p}^n_t} - \dual{\bm{p}^n_t} } \cvgas[\mathbb{P}^\omega]{n} 0.
	\end{align*}
	It follows that $ ( f(\dual{\bm{p}^n_t}), t \in [0,T] ) $ is tight in $ \mathcal{C}([0,T], \R) $ for any $ \phi \in \mathcal{C}^\infty_c (\R) $ and $ f \in C^2 (\R) $.
	By Theorem~3.9.1 in \cite{ethier_markov_1986}, the sequence $ ( \bm{p}^n_t, t \in [0,T] ) $ is tight in $ \mathcal{C}([0,T], \Xi ) $.
\end{proof}

\subsection{Continuity estimate} \label{subsec:holder_estimate}

To prove Lemma~\ref{lemma:holder_pn}, we need the following bounds, which will be proved in Appendix~\ref{appendix:heat_kernel_estimates}.
Recall the definition of $g^\omega_t(x, y)$ in \eqref{def_g_omega}.

\begin{lemma} \label{lemma:continuity_g_omega}
	There exist constants $C, C' > 0$ such that, for $ \mu $-almost every $\omega \in \Omega$, for all $n \geq 1$, $ t \geq 0 $ and $ x, y \in \Z $,
	\begin{align*}
	& \sqrt{n} \sum_{z \in \Z} \int_0^t \left( g^\omega_{ns}(x,z) - g^\omega_{ns}(y,z) \right)^2 ds \leq C t^{1/4} \abs{\frac{x-y}{\sqrt{n}}}^{1/2} \\
	& \sup_{z \in \Z} \abs{ g^\omega_t(x,z) - g^\omega_t(y,z) } \leq C' t^{-3/4} \abs{x-y}^{1/2}.
	\end{align*}
\end{lemma}

\begin{proof}[Proof of Lemma~\ref{lemma:holder_pn}]
	We begin by noting that $ p^n_{nt} $ has the integral form representation
	\begin{align*}
	p^n_{nt}(\omega,x) = \sum_{y \in \Z} g^\omega_{nt}(x,y) p^n_0(y) + \sum_{y \in \Z} \int_0^t g^\omega_{ns}(x,y) \sqrt{ \frac{\sqrt{n}}{N(\omega,y)} p^n_{ns}(\omega,y) (1 - p^n_{ns}(\omega,y) )} d B^y_s.
	\end{align*}
	Then, for $k \geq 1$,
	\begin{multline} \label{2k_moment}
	\Eq{ \abs{ p^n_{nt}(\omega,x) - p^n_{nt}(\omega,y) }^{2k} } \leq 2^{2k-1} \abs{ \sum_{z\in \Z} g^\omega_{nt}(x,z) p^n_0(z) - \sum_{z \in \Z} g^\omega_{nt}(y,z) p^n_0(z) }^{2k} \\ + 2^{2k-1} \Eq{ \abs{ \sum_{z \in \Z} \int_0^t \left( g^\omega_{ns}(x,z) - g^\omega_{ns}(y,z) \right) \sqrt{ \frac{\sqrt{n}}{N(\omega,z)} p^n_{ns}(\omega,z) ( 1 - p^n_{ns}(\omega,z) ) } d B^z_s }^{2k} }.
	\end{multline}
	We then bound each term on the right hand side separately, starting with the second one.
	By the Burkholder-Davis-Gundy inequality, \eqref{uniform_ellipticity} and since $ 0 \leq p^n_{nt}(\omega,z) \leq 1 $, for some $ C > 0 $,
	\begin{multline} \label{kolmogorov1}
	\Eq{ \abs{ \sum_{z \in \Z} \int_0^t \left( g^\omega_{ns}(x,z) - g^\omega_{ns}(y,z) \right) \sqrt{ \frac{\sqrt{n}}{N(\omega,z)} p^n_{ns}(\omega,z) ( 1 - p^n_{ns}(\omega,z) ) } d B^z_s }^{2k} } \\ 
	\begin{aligned}
	& \leq C \left( \sqrt{n} \sum_{z \in \Z} \int_0^t \left( g^\omega_{ns}(x,z) - g^\omega_{ns}(y,z) \right)^2 ds \right)^k \\
	& \leq C t^{k/4} \abs{\frac{x-y}{\sqrt{n}}}^{k/2},
	\end{aligned}
	\end{multline}
	where we used Lemma~\ref{lemma:continuity_g_omega} in the last line.
	For the first term, we combine two bounds.
	First note that, for all $ R > 0 $, $ x, y \in \Z $
	\begin{multline*}
	\abs{ \sum_{z\in \Z} g^\omega_{nt}(x,z) p^n_0(z) - \sum_{z \in \Z} g^\omega_{nt}(y,z) p^n_0(z) } \leq \sum_{z \in B\left( \frac{x+y}{2}, \frac{\abs{x-y}}{2} + R \sqrt{n} \right) \cap \Z} \abs{ g^\omega_{nt}(x,z) - g^\omega_{nt}(y,z) } \\ + \sum_{z \in \Z \setminus B(x, R\sqrt{n})} g^\omega_{nt}(x,z) + \sum_{z \in \Z \setminus B(y,R\sqrt{n})} g^\omega_{nt}(y,z). 
	\end{multline*}
	By Lemma~\ref{lemma:bound_variance_RW}, the last two terms are both bounded by $ \frac{C}{R^2} $ for some constant $ C > 0 $, for all $n \geq 1$.
	Then by Lemma~\ref{lemma:continuity_g_omega},
	\begin{align*}
	\abs{ \sum_{z\in \Z} g^\omega_{nt}(x,z) p^n_0(z) - \sum_{z \in \Z} g^\omega_{nt}(y,z) p^n_0(z) } \leq C \left( R + \frac{\abs{x-y}}{2 \sqrt{n}} \right) \sqrt{n} (nt)^{-3/4} \abs{x-y}^{1/2} + \frac{2 C}{R^2}.
	\end{align*}
	Choosing $ R = \abs{\frac{x-y}{\sqrt{n}}}^{-1/6} $, we obtain that for all $ A > 0 $ there exists a constant $ C > 0 $ such that, for all $ x, y \in \Z $ such that $ \abs{\frac{x-y}{\sqrt{n}}} \leq A $,
	\begin{align} \label{holder1}
	\abs{ \sum_{z\in \Z} g^\omega_{nt}(x,z) p^n_0(z) - \sum_{z \in \Z} g^\omega_{nt}(y,z) p^n_0(z) } \leq C ( 1 + t^{-3/4} ) \abs{ \frac{x-y}{\sqrt{n}} }^{1/3}.
	\end{align}
	This bound doesn't behave well when $t$ is too small. To remedy this, we note that by assumption \eqref{holder_initial_condition},
	\begin{align*}
	\abs{ \sum_{z \in \Z} g^\omega_{nt}(x,z) p^n_0(z) - p^n_0(x) } & \leq C \sum_{z \in \Z} g^\omega_{nt}(x,z) \abs{ \frac{x-y}{\sqrt{n}} } \\
	& \leq C \, \Eq[x]{ \abs{ \frac{\xi_{nt} - x}{\sqrt{n}} }^2 }^{1/2} \\
	& \leq C \sqrt{t},
	\end{align*}
	where we have used Jensen's inequality in the second line and Lemma~\ref{lemma:bound_variance_RW} in the last line.
	Hence, using assumption \eqref{holder_initial_condition} again,
	\begin{align*}
	\abs{ \sum_{z \in \Z} g^\omega_{nt}(x,z) p^n_0(z) - \sum_{z \in \Z} g^\omega_{nt}(y,z) p^n_0(z) } & \leq 2 C \sqrt{t} + \abs{ p^n_0(x) - p^n_0(y) } \\
	& \leq 2 C \sqrt{t} + C \abs{\frac{x-y}{\sqrt{n}}}. \numberthis \label{holder2}
	\end{align*}
	The above inequality behaves well when $ t $ is not too large.
	Hence, using \eqref{holder1} when $ t \geq \abs{\frac{x-y}{\sqrt{n}}}^{4/15} $ and \eqref{holder2} when $ t \leq \abs{\frac{x-y}{\sqrt{n}}}^{4/15} $, we conclude that there exists a constant $ C > 0 $ such that, for all $ t \geq 0 $ and every $ x, y \in \Z $,
	\begin{align} \label{kolomogorov2}
	\abs{ \sum_{z \in \Z} g^\omega_{nt}(x,z) p^n_0(z) - \sum_{z \in \Z} g^\omega_{nt}(y,z) p^n_0(z) } \leq C \abs{\frac{x-y}{\sqrt{n}}}^{2/15}.
	\end{align}
	Finally, plugging \eqref{kolmogorov1} and \eqref{kolomogorov2} into \eqref{2k_moment} and noting that the left hand side stays bounded even when $ \abs{x-y} $ grows, we obtain that there exists a constant $ C > 0 $ depending on $ k \geq 1 $ such that, for $ t \in [0,T] $, for all $ x, y \in \Z $,
	\begin{align*}
	\Eq{\abs{p^n_{nt}(\omega,x) - p^n_{nt}(\omega,y)}^{2k}} \leq C \abs{\frac{x-y}{\sqrt{n}}}^{4k/15} \quad \mu(d\omega)-a.s.
	\end{align*}
	By Kolmogorov's continuity theorem, this implies that for any $ \beta \in (0, 2/15) $ there exist random variables $ (C_n, n \geq 1) $ such that, for $ n \geq 1 $ and $ T > 0 $, and every $ x, y \in \Z $,
	\begin{align*}
	\sup_{t \in [0,T]}  \abs{p^n_{nt}(\omega,x) - p^n_{nt}(\omega,y)} \leq C_n \abs{\frac{x-y}{\sqrt{n}}}^{\beta}
	\end{align*}
	almost surely and $ \Eq{C_n^2} \leq C $ for all $ n \geq 1 $, $ \mu(d\omega) $-almost surely.
	This concludes the proof of Lemma~\ref{lemma:holder_pn}.
\end{proof}

\appendix

\section{Heat kernel estimates} \label{appendix:heat_kernel_estimates}

To prove Lemma~\ref{lemma:continuity_g_omega} and Theorem~\ref{thm:local_clt}, we need various estimates on the heat kernel associated to the random walk $ (\xi_t)_{t \geq 0} $.
Let us define, for $ t \geq 0 $ and $ x, y \in \Z $,
\begin{align*}
h^\omega_t(x,y) = \frac{g^\omega_t(x,y)}{\pi(\omega,y)}.
\end{align*}
We follow a similar approach to that of \cite{derrien_local_2015}, using \eqref{uniform_ellipticity} to improve the bounds.
We state two lemmas and we give their proofs.
At the end of the section, we prove Lemma~\ref{lemma:continuity_g_omega}.

\begin{lemma} \label{lemma:bound_h}
	There exists a constant $ C > 0 $ such that, for $ \mu $-almost every $ \omega \in \Omega $,
	\begin{align*}
	\sup_{t \geq 0} \sup_{x, y \in \Z} \sqrt{t} \, h^\omega_t(x,y) \leq C.
	\end{align*}
\end{lemma}

\begin{lemma} \label{lemma:continuity_h}
	There exists a constant $ C > 0 $ such that, for $ \mu $-almost every $ \omega \in \Omega $,
	\begin{align*}
	\sup_{t \geq 0} \sup_{x, y, z \in \Z} t^{3/4} \frac{\abs{h^\omega_t(x,y) - h^\omega_t(x,z)}}{\abs{y-z}^{1/2}} \leq C.
	\end{align*}
\end{lemma}

\begin{proof}[Proof of Lemma~\ref{lemma:bound_h}]
	For $ f \in \mathcal{F}_0(\Z) $, we set
	\begin{align*}
	Q^\omega(f) = \frac{1}{2} \sum_{x \in \Z} \sum_{z \in \{-1,1\}} \pi(\omega,x) j(\omega,x,z) \left( f(x+z) - f(x) \right)^2.
	\end{align*}
	Using Bernstein's representation theorem as in \cite{derrien_local_2015}, one shows that for all $ x \in \Z $, $ t > 0 $,
	\begin{align} \label{bound_Q}
	Q^\omega(h^\omega_t(x, \cdot)) \leq \frac{e^{-1}}{t} h^\omega_t(x,x).
	\end{align}
	We then proceed as in the proof of Theorem~3.3 in \cite{derrien_local_2015}.
	Fix $ x \in \Z $, $ t > 0 $ and $ k \geq 1 $ (to be chosen later on).
	Let $ y_0 \in B(x,k) \cap \Z $ be such that
	\begin{align*}
	h^\omega_t(x,y_0) = \min_{y \in B(x,k) \cap \Z} h^\omega_t(x,y).
	\end{align*}
	Then
	\begin{align*}
	h^\omega_t(x,y_0) & \leq \frac{1}{\sum_{y \in B(x,k) \cap \Z} \pi(\omega,y)} \sum_{y \in B(x,k) \cap \Z} h^\omega_t(x,y) \pi(\omega,y) \\
	& \leq \left( \sum_{y \in B(x,k) \cap \Z} \pi(\omega,y) \right)^{-1},
	\end{align*}
	since $ \sum_{y \in \Z} h^\omega_t(x,y)\pi(\omega,y) = 1 $.
	Hence, assuming without loss of generality that $ y_0 \geq x $,
	\begin{align*}
	h^\omega_t(x,x) - \left( \sum_{y \in B(x,k) \cap \Z} \pi(\omega,y) \right)^{-1} & \leq h^\omega_t(x,x) - h^\omega_t(x,y_0) \\
	& \leq \sum_{l=x}^{y_0-1} \abs{h^\omega_t(x,l) - h^\omega_t(x,l+1)}.
	\end{align*}
	By the Cauchy-Schwarz inequality, the right hand side is bounded by
	\begin{multline*}
	\left( \sum_{l=x}^{y_0-1} \pi(\omega,l) j(\omega,l,1) \left( h^\omega_t(x,l+1) - h^\omega_t(x,l) \right)^2 \right)^{1/2} \left( \sum_{l=x}^{y_0-1} \frac{1}{\pi(\omega,l) j(\omega,l,1)} \right)^{1/2} \\
	\leq \sqrt{2} Q^\omega(h^\omega_t(x,\cdot))^{1/2} \left( \sum_{l \in B(x,k) \cap \Z} \frac{1}{\pi(\omega,l) j(\omega,l,1)} \right)^{1/2}.
	\end{multline*}
	By \eqref{bound_Q} and \eqref{uniform_ellipticity}, recalling the expression of $ \pi(\omega,l) $ and $ j(\omega,l,1) $,
	\begin{align} \label{bound_h1}
	h^\omega_t(x,x) - \left( \sum_{y \in B(x,k) \cap \Z} \pi(\omega,y) \right)^{-1} \leq 2\sqrt{3} K^2 \sqrt{k} \frac{e^{-1/2}}{\sqrt{t}} h^\omega_t(x,x)^{1/2}.
	\end{align}
	Since the random walk $ (\xi_t)_{t\geq 0} $ is null recurrent, $ h^\omega_t(x,x) \to 0 $ as $ t \to \infty $ for all $ x \in \Z $.
	In addition, $ h^\omega_t(x,x) \leq \pi(\omega,x)^{-1} $, hence we can always choose $ k \geq 1 $ such that
	\begin{align} \label{choose_k}
	\sum_{y \in B(x,k-1) \cap \Z} \pi(\omega,y) \leq \frac{3}{h^\omega_t(x,x)} \leq \sum_{y\in B(x,k) \cap \Z} \pi(\omega,y).
	\end{align}
	Using such a $ k $ in \eqref{bound_h1} results in
	\begin{align*}
	\frac{2}{3} h^\omega_t(x,x) & \leq 2 \sqrt{3} K^2 \frac{e^{-1/2}}{\sqrt{t}} \left( k h^\omega_t(x,x) \right)^{1/2} \\
	& \leq 6 K^2 \frac{e^{-1/2}}{\sqrt{t}} \left( \frac{1}{k} \sum_{y \in B(x,k-1) \cap \Z} \pi(\omega,y) \right)^{-1/2}.
	\end{align*}
	Finally, by \eqref{uniform_ellipticity},
	\begin{align*}
	\frac{1}{k} \sum_{y \in B(x,k-1) \cap \Z} \pi(\omega,y) \geq \frac{1}{K^4} \frac{2k-1}{k} \geq \frac{1}{K^4}.
	\end{align*}
	As a result, for all $ x \in \Z $, $ t > 0 $,
	\begin{align} \label{result_x}
	\sqrt{t} h^\omega_t(x,x) \leq 9 K^4 e^{-1/2} \quad \mu(d\omega)-a.s.
	\end{align}
	This proves the result when $ x = y $, let us now extend it for all $ x, y \in \Z $.
	
	Reasoning as above, we obtain for $ k \geq 1 $,
	\begin{align*}
	h^\omega_t(x,y) - \left( \sum_{z \in B(y,k) \cap \Z} \pi(\omega,z) \right)^{-1} \leq 2 \sqrt{3 k} K^2 Q^\omega(h^\omega_t(x,\cdot))^{1/2}.
	\end{align*}
	By \eqref{bound_Q} and \eqref{result_x},
	\begin{align*}
	Q^\omega(h^\omega_t(x, \cdot))^{1/2} & \leq \frac{e^{-1/2}}{\sqrt{t}} h^\omega_t(x,x)^{1/2} \\
	& \leq \frac{C}{t^{3/4}}, \numberthis \label{boundQ2}
	\end{align*}
	for some constant $ C > 0 $.
	We then choose $ k \geq 1 $ such that
	\begin{align*}
	\sum_{z \in B(y,k-1) \cap \Z} \pi(\omega,z) \leq \frac{3}{h^\omega_t(x,y)} \leq \sum_{z \in B(y,k) \cap \Z} \pi(\omega,z).
	\end{align*}
	Noting that
	\begin{align*}
	2k - 1 \leq K^4 \frac{3}{h^\omega_t(x,y)},
	\end{align*}
	we obtain
	\begin{align*}
	\frac{2}{3} h^\omega_t(x,y) \leq 2 \sqrt{3} K^2 \frac{C}{t^{3/4}} \left( \frac{3}{2} \frac{K^4}{h^\omega_t(x,y)} + \frac{1}{2} \right)^{1/2}.
	\end{align*}
	Multiplying by $ t^{3/4} \sqrt{h^\omega_t(x,y)} $ on both sides, we obtain the result after noting that $ h^\omega_t(x,y) \leq K^4 $.
\end{proof}

We now turn to the proof of Lemma~\ref{lemma:continuity_h}.

\begin{proof}[Proof of Lemma~\ref{lemma:continuity_h}]
	For $ x, y, z \in \Z $ and $ t > 0 $, we write, assuming for example that $ z \geq y $,
	\begin{align*}
	\abs{h^\omega_t(x,y) - h^\omega_t(x,z)} \leq \sum_{l = y}^{z-1} \abs{h^\omega_t(x,l) - h^\omega_t(x,l+1)}.
	\end{align*}
	Reasoning as in the proof of Lemma~\ref{lemma:bound_h}, we obtain
	\begin{align*}
	\abs{h^\omega_t(x,y) - h^\omega_t(x,z)} \leq \sqrt{2} Q^\omega(h^\omega_t(x, \cdot))^{1/2} \left( \sum_{l=y}^{z-1} \frac{1}{\pi(\omega,l) j(\omega,l,1)} \right)^{1/2}.
	\end{align*}
	By \eqref{uniform_ellipticity}, for some $ C > 0 $
	\begin{align*}
	\sum_{l=y}^{z-1} \frac{1}{\pi(\omega,l) j(\omega,l,1)} \leq C \abs{x-y}.
	\end{align*}
	Lemma~\ref{lemma:continuity_h} then follows using \eqref{boundQ2}.
\end{proof}

Finally, we use Lemma~\ref{lemma:continuity_h} to prove Lemma~\ref{lemma:continuity_g_omega}.

\begin{proof}[Proof of Lemma~\ref{lemma:continuity_g_omega}]
	We first note that the second part of the statement follows at once from Lemma~\ref{lemma:continuity_h} and \eqref{uniform_ellipticity}, noting that $ h^\omega_t(x,y) = h^\omega_t(y,x) $ by reversibility.
	For the first statement, we have, using Lemma~\ref{lemma:continuity_h},
	\begin{align*}
	\left( g^\omega_{ns}(x,z) - g^\omega_{ns}(y,z) \right)^{2} \leq C (ns)^{-3/4} \abs{x-y}^{1/2} \left( g^\omega_{ns}(x,z) + g^\omega_{ns}(y,z) \right).
	\end{align*}
	Since $ \sum_{z \in \Z} g^\omega_t(\cdot,z) = 1 $, summing over $ z $ and integrating, we obtain
	\begin{align*}
	\sqrt{n} \sum_{z \in\Z} \int_{0}^{t} \left( g^\omega_{ns}(x,z) - g^\omega_{ns}(y,z) \right)^2 ds \leq C \abs{\frac{x-y}{\sqrt{n}}}^{1/2} \int_{0}^{t}s^{-3/4}ds,
	\end{align*}
	and the result follows.
\end{proof}

\section{The local central limit theorem for reversible random walks in a random environment} \label{sec:local_clt}

Armed with Lemma~\ref{lemma:bound_h} and Lemma~\ref{lemma:continuity_h}, we can prove the local central limit theorem for the random walk $ (\xi_t)_{t \geq 0} $ (Theorem~\ref{thm:local_clt}).
Again, we follow the method used in \cite{derrien_local_2015}.

\begin{proof}[Proof of Theorem~\ref{thm:local_clt}]
	For $ \delta > 0 $ and $ x, y \in \Z $, we write
	\begin{align*}
	\sqrt{n} \frac{g^\omega_{nt}(x,y)}{\pi(\omega,y)} - G_t &\left( \frac{x-y}{\sqrt{n}} \right) = \sqrt{n} h^\omega_{nt}(x,y) \left( 1 - \frac{1}{2 \delta \sqrt{n}} \sum_{k \in B(y, \delta \sqrt{n}) \cap \Z} \pi(\omega,k) \right) \\
	& \quad + \frac{1}{2\delta} \sum_{k \in B(y,\delta \sqrt{n}) \cap \Z} \pi(\omega,k) \left( h^\omega_{nt}(x,y) - h^\omega_{nt}(x,k) \right) \\
	& \quad + \frac{1}{2\delta} \left( \Pq[x]{\xi_{nt} \in B(y, \delta \sqrt{n}) \cap \Z} - \P{ \mathcal{N}(x/\sqrt{n}, \sigma^2 t) \in B(y/\sqrt{n}, \delta)} \right) \\
	& \quad + \frac{1}{2\delta} \int_{y/\sqrt{n} - \delta}^{y/\sqrt{n} + \delta} G_t\left( \frac{x}{\sqrt{n}} - z \right) dz - G_t\left( \frac{x-y}{\sqrt{n}} \right) \\
	& = A^\delta_n + B^\delta_n + C^\delta_n + D^\delta_n.
	\end{align*}
	We then bound each term separately.
	
	\paragraph*{Bound on $ A^\delta_n $}
	
	By Lemma~\ref{lemma:bound_h},
	\begin{align} \label{bound_An}
	\sup_{t \in [0,T]} \max_{\substack{x \in B(0,R\sqrt{n}) \cap \Z \\ y \in B(0,R \sqrt{n}) \cap \Z}} \abs{A^\delta_n} \leq C \max_{y \in B(0,r\sqrt{n}) \cap \Z} \abs{1 - \Pi^{n, \delta}\left( \omega, \frac{y}{\sqrt{n}} \right)},
	\end{align}
	where $ \Pi^{n, \delta} $ was defined in \eqref{definition_Pi}.
	By Lemma~\ref{lemma:uniform_cvg_pi}, the right hand side converges to 0 as $ n \to \infty $ for any $ \delta > 0 $, $ \mu(d\omega) $-almost surely.
	
	\paragraph*{Bound on $ B^\delta_n $}
	
	From Lemma~\ref{lemma:continuity_h} and \eqref{uniform_ellipticity},
	\begin{align*}
	\abs{B^\delta_n} & \leq \frac{K^4}{2\delta} \sum_{k \in B(y, \delta \sqrt{n}) \cap \Z} C (nt)^{-3/4} \abs{y-k}^{1/2} \\
	& \leq C \frac{K^4}{t^{3/4}} \delta^{1/2}
	\end{align*}
	Hence for all $ \delta > 0 $ and $ 0 < \varepsilon < T $,
	\begin{align} \label{bound_Bn}
	\lim_{n \to \infty} \sup_{t \in [\varepsilon, T]} \max_{x, y \in \Z} \abs{B^\delta_n} = 0 \quad \mu(d\omega)-a.s.
	\end{align}
	
	\paragraph*{Bound on $ C^\delta_n $}
	
	For $ t \geq 0 $ set
	\begin{align*}
	& F_n^\omega(t,x,y) = \Pq[\lfloor\sqrt{n}x \rfloor]{\frac{1}{\sqrt{n}} \xi_{nt} \leq y} \\
	& F(t,x,y) = \P{\mathcal{N}(x, \sigma^2 t) \leq y}.
	\end{align*}
	By the central limit theorem, $ F_n^\omega $ converges pointwise to $ F $ as $ n \to \infty $, for $ \mu $ almost every $ \omega \in \Omega $.
	Moreover, by the functional central limit theorem (Theorem~\ref{thm:functional_clt}), $ F_n^\omega(\cdot, x, y) $ converges to $ F(\cdot, x, y) $ uniformly on $ [0,T] $ for each $ x, y \in \R $ (using Fatou's lemma).
	Also, $ x \mapsto F_n^\omega(t,x,y) $ and $ y \mapsto F_n^\omega(t,x,y) $ are both monotone (and so are $ x \mapsto F(t,x,y) $ and $ y \mapsto F(t,x,y) $), and $ F $ is continuous.
	It then follows that $ F_n^\omega $ converges to $ F $ uniformly on compact sets of $ \R_+ \times \R^2 $, $ \mu(d\omega) $ almost surely.
	
	Now note that
	\begin{multline*}
	C^\delta_n = \frac{1}{2\delta} \left[ \left( F_n^\omega\left(t,\frac{x}{\sqrt{n}}, \frac{y}{\sqrt{n}} + \delta \right) - F\left(t,\frac{x}{\sqrt{n}},\frac{y}{\sqrt{n}} + \delta \right) \right) \right. \\ \left. - \left( F_n^\omega\left( t, \frac{x}{\sqrt{n}}, \frac{y}{\sqrt{n}} - \delta \right) - F\left( t, \frac{x}{\sqrt{n}}, \frac{y}{\sqrt{n}} - \delta \right) \right) \right].
	\end{multline*}
	Hence
	\begin{align*}
	\sup_{t \in [0,T]} \max_{\substack{x \in B(0,R\sqrt{n}) \cap \Z \\ y \in B(0,R\sqrt{n}) \cap \Z}} \abs{C^\delta_n} \leq \frac{1}{\delta} \sup_{t \in [0,T]} \sup_{\substack{x \in B(0,R) \\ y \in B(0,R + \delta)}} \abs{F_n^\omega(t,x,y) - F(t,x,y)}.
	\end{align*}
	By the uniform convergence of $ F_n^\omega $ on compact sets, we thus have, for any $ \delta > 0 $,
	\begin{align} \label{bound_Cn}
	\lim_{n \to \infty} \sup_{t \in [0,T]} \max_{\substack{x \in B(0,R\sqrt{n}) \cap \Z \\ y \in B(0,R\sqrt{n}) \cap \Z}} \abs{C^\delta_n} = 0 \quad \mu(d\omega)-a.s.
	\end{align}
	
	\paragraph*{Bound on $ D^\delta_n $}
	
	Since
	\begin{align*}
	\sup_{x \in \R} \abs{\partial_x G_t(x)} = \frac{e^{-1/2}}{\sqrt{2\pi} \sigma^2 t},
	\end{align*}
	we have, for all $ 0 < \varepsilon < T $ and $ \delta > 0 $,
	\begin{align} \label{bound_Dn}
	\sup_{n \geq 1} \sup_{t \in [\varepsilon, T]} \sup_{x, y \in \R} \abs{D^\delta_n} \leq \frac{e^{-1/2}}{\sqrt{2\pi} \sigma^2 \varepsilon} \delta.
	\end{align}
	
	Combining \eqref{bound_An}, \eqref{bound_Bn}, \eqref{bound_Cn} and \eqref{bound_Dn}, we see that $ \delta $ can be chosen so that $ \abs{B^\delta_n} $ and $ \abs{D^\delta_n} $ are arbitrarily small for all $ n \geq 1 $, and then we can choose $ n_0 $ such that for all $ n \geq n_0 $, $ \abs{A^\delta_n} $ and $ \abs{C_n^\delta} $ are arbitrarily small.
	This concludes the proof of Theorem~\ref{thm:local_clt}.
\end{proof}

\section{Proof of Lemma~\ref{lemma:coupling}} \label{appendix:coupling}

\begin{proof}[Proof of Lemma~\ref{lemma:coupling}]
	First of all, we can assume without loss of generality that $ \E{Y} = 1 $.
	Then there exists a random variable $ \tilde{X} $ taking values in $ \R_+ $ such that, for any measurable non-negative function $ f $,
	\begin{align*}
	\mathbb{E}[ f(\tilde{X}) ] = \E{ f(X) Y }.
	\end{align*}
	Let us show that there exists a coupling of $ X $ and $ \tilde{X} $ such that $ \tilde{X} \geq X $ almost surely.
	
	Let $ \mathcal{P}_X $ denote the law of the random variable $ X $ and set, for all $ x \in [0,+\infty] $,
	\begin{align*}
	F(x) = \int_{0}^{x} \mathcal{P}_X(du), && G(x) = \int_{0}^{x} g(u) \mathcal{P}_X(du).
	\end{align*}
	We prove that for all $ x \geq 0 $,
	\begin{align} \label{inequality_repartition_functions}
	G(x) \leq F(x).
	\end{align}
	Define $ x_0 \in [0,+\infty] $ by
	\begin{align*}
	x_0 = \inf \lbrace x \geq 0 : g(x) > 1 \rbrace.
	\end{align*}
	Then \eqref{inequality_repartition_functions} clearly holds for all $ x < x_0 $.
	If $ x_0 = +\infty $, then \eqref{inequality_repartition_functions} is proved.
	Otherwise, suppose that there exists $ x \geq x_0 $ such that
	\begin{align*}
	G(x) > F(x).
	\end{align*}
	Since $ g $ is non-decreasing, $ g(u) \geq 1 $ for all $ u > x_0 $.
	If $ g(x_0) < 1 $, \eqref{inequality_repartition_functions} holds for $ x = x_0 $ and we can assume that $ x > x_0 $, otherwise we have $ g(u) \geq 1 $ for all $ u \geq x_0 $.
	As a result,
	\begin{align*}
	\int_{x}^{+\infty} g(u) \mathcal{P}_X(du) \geq \int_{x}^{+\infty} \mathcal{P}_X(du).
	\end{align*}
	Adding this to the previous inequality, we obtain
	\begin{align*}
	G(+\infty) > F(+\infty).
	\end{align*}
	This is a contradiction because we have assumed $ \E{Y} = G(+\infty) = 1 = F(+\infty) $.
	We have thus proved \eqref{inequality_repartition_functions} for all $ x \geq 0 $.
	Since $ F $ (resp. $ G $) is the cumulative distribution function of $ X $ (resp. of $ \tilde{X} $), there exists a coupling of $ X $ and $ \tilde{X} $ such that $ \tilde{X} \geq X $ almost surely (for example let $ U $ be a uniform random variable on $ [0,1] $ and set $ X = F^{-1}(U) $ and $ \tilde{X} = G^{-1}(U) $, where $ F^{-1} $ and $ G^{-1} $ are the right continuous generalised inverse functions of $ F $ and $ G $).
	It then follows that
	\begin{align*}
	\E{X Y} = \mathbb{E} [\tilde{X}] \geq \E{X} = \E{X}\E{Y},
	\end{align*}
	and the Lemma is proved.
\end{proof}

\end{document}